\documentclass[11pt]{article}
\usepackage{amssymb,amsmath,color,amsfonts,float,amscd,amsthm,bm,mathrsfs} 
\usepackage{amsmath}

\allowdisplaybreaks[4]

\catcode`!=11
\let\!int\int \def\int{\displaystyle\!int}
\let\!lim\lim \def\lim{\displaystyle\!lim}
\let\!sum\sum \def\sum{\displaystyle\!sum}
\let\!sup\sup \def\sup{\displaystyle\!sup}
\let\!inf\inf \def\inf{\displaystyle\!inf}
\let\!cap\cap \def\cap{\displaystyle\!cap}
\let\!max\max \def\max{\displaystyle\!max}
\let\!min\min \def\min{\displaystyle\!min}

\let\oldsection\section
\renewcommand\section{\setcounter{equation}{0}\oldsection}

\newtheorem{theorem}{Theorem}[section]
\newtheorem{lemma}[theorem]{Lemma}

\theoremstyle{definition}

\theoremstyle{remark}

\allowdisplaybreaks

\begin{document}

\title{Vanishing viscosity limit for compressible magnetohydrodynamics equations with transverse background magnetic field}

\author{Xiufang Cui\footnote{School of Mathematical Sciences, Shanghai Jiao Tong University, Shanghai 200240, P. R. China. Email: cuixiufang@sjtu.edu.cn}
\and
Shengxin Li\footnote{School of Mathematical Sciences, Shanghai Jiao Tong University, Shanghai 200240, P. R. China. Email: lishengxin@sjtu.edu.cn}
\and
Feng Xie
\footnote{School of Mathematical Sciences, CMA-Shanghai, and MOE-LSC, Shanghai Jiao Tong University, Shanghai 200240, P. R. China.
Email: tzxief@sjtu.edu.cn}
}
\date{}
\maketitle

\noindent{\bf Abstract:}
 We are concerned with the uniform regularity estimates and vanishing viscosity limit of solution to two dimensional viscous compressible magnetohydrodynamics (MHD) equations with transverse background magnetic field. When the magnetic field is assumed to be transverse to the boundary and the tangential component of magnetic field satisfies zero Neumann boundary condition, even though the velocity is imposed the no-slip boundary condition, the uniform regularity estimates of solution and its derivatives still can be achieved in suitable conormal Sobolev spaces in the half plane $\mathbb{R}^2_+$, and then the vanishing viscosity limit is justified in $L^\infty$ sense based on these uniform regularity estimates and some compactness arguments. At the same time, together with \cite{CLX21}, our results show that the transverse background magnetic field can prevent the strong boundary layer from occurring for compressible magnetohydrodynamics whether there is magnetic diffusion or not.
\vskip0.2cm

%
%
\maketitle





\section{Introduction}In this paper we consider the vanishing viscosity limit of solution to
two dimensional viscous compressible  magnetohydrodynamics (MHD) equations in the half plane $\mathbb{R}_+^2:=\{(x,y)|x\in\mathbb{R},\ y\geq 0\}$:
\begin{align}\label{1.1}
\begin{cases}
\partial_t\rho^\varepsilon+\nabla\cdot(\rho^\varepsilon {\bf{v}}^\varepsilon)=0,\\
\partial_t(\rho^\varepsilon {\bf{v}}^\varepsilon)+\mathrm{div}(\rho^\varepsilon {\bf{v}}^\varepsilon\otimes {\bf{v}}^\varepsilon)-\varepsilon\mu\Delta {\bf{v}}^\varepsilon-\varepsilon(\mu+\lambda)\nabla(\nabla\cdot {\bf{v}}^\varepsilon)+\nabla p ^{\varepsilon}=(\nabla\times {\bf{B}}^\varepsilon)\times {\bf{B}}^\varepsilon,\\
\partial_t{\bf{B}}^\varepsilon-\nabla\times({\bf{v}}^\varepsilon\times {\bf{B}}^\varepsilon)=\varepsilon\kappa \Delta {\bf{B}}^\varepsilon,\quad \mathrm{div}\  {\bf{B}}^\varepsilon=0,
\end{cases}
\end{align}	
where $\rho^\varepsilon$ is the density, ${\bf{v}}^\varepsilon=(v_1^\varepsilon, v_2^\varepsilon)$ denotes the velocity and ${\bf{B}}^\varepsilon=(b_1^\varepsilon, b_2^\varepsilon)$ stands for the magnetic field.  The viscosity coefficients  $\varepsilon \mu,$  $\varepsilon\lambda$  and the magnetic diffusion coefficient $\varepsilon\kappa$ are assumed to be the same order in term of a small parameter $\varepsilon$ with $\mu>0$ and $\mu+\lambda>0$. The operators $\nabla=(\partial_x, \partial_y)$ and $\Delta=\partial_x^2+\partial_y^2$. The pressure $p^\varepsilon$ is a function of $\rho^\varepsilon$,  which takes the following form:
\begin{align}\label{1.2}
p^\varepsilon=(\rho^\varepsilon)^\gamma,\quad \gamma\geq1,
\end{align}
where $\gamma$ is the adiabatic constant. The initial data is given by
\begin{align*}
(\rho^\varepsilon, {\bf{v}}^\varepsilon, {\bf{B}}^\varepsilon)(t, x, y)|_{t=0}=(\rho^\varepsilon_0, {\bf{v}}^\varepsilon_0,  {\bf{B}}^\varepsilon_0)(x, y).
\end{align*}	
The no-slip boundary condition is imposed on the velocity field:
\begin{align}\label{1.3}
{\bf{v}}^\varepsilon|_{y=0}=0.
\end{align}	
The main goal of this paper is to analyze the effect of transverse background magnetic field on the vanishing viscosity limit process of solution to (\ref{1.1})-(\ref{1.3}). Consequently, we impose the following boundary conditions on magnetic field.
\begin{align}\label{1.4}
\partial_y b_1^\varepsilon|_{y=0}=0,\quad  b_2^\varepsilon|_{y=0}=1.
\end{align}

As is well-known that the system of magnetohydrodynamics (MHD) equations is an important model in plasma physics, and also attracts many attentions from mathematicians. Extensive works exist for the study of compressible MHD equations \cite{CW02, CW03, DF06, L15, HT05, Wang03} and incompressible MHD equations  \cite{AZ17, CRW13, CW11, LZ14, LXZ15, ST83}.

The inviscid limit problem is also an important but challenging problem in both hydrodynamics and applied mathematics, see \cite{AD04, CP10, CPW95, HWY12, RWX14, RXZ16}. Particularly,  when the inviscid limit process is considered in a domain with boundaries, it becomes much more challenging due to the possible presence of strong boundary layers \cite{Ole, prandtl, SH60}.

However, when both velocity and magnetic field are imposed the Navier-slip boundary conditions, the strong boundary layer usually disappears. Thus, under this kind of slip boundary conditions, it is reasonable to justify the inviscid limit of solution to incompressible MHD system directly without studying the boundary layers, see \cite{DXX20, WZ17,  XXW09}.
But, when the velocity is given the no-slip boundary condition, in general the strong boundary layer always occurs. At least, it is the case for the Navier-Stokes equations \cite{GGW,LWXY,O,Ole,prandtl,WXY}. Consequently, due to the  appearance of strong boundary layer,  the inviscid limit in $L^\infty$ sense becomes dramatically difficult as the viscosity coefficient goes to zero. And the essential difficulty is down to uncontrollability of the vorticity of boundary layer.
Recently, although the velocity is imposed the no-slip boundary condition, Liu, the third author and Yang not only established the well-posedness of solution to MHD boundary layer equations but also proved the validity of Prandtl boundary layer expansion in the Sobolev spaces under the condition that the tangential component of magnetic filed does not degenerate near the physical boundary initially in \cite{LXY191, LXY192}. Where the tangential component of magnetic field plays a key role in the stability of boundary layers and vanishing viscosity limit process.  Thereafter, under the no-slip boundary condition on velocity, Wang and the third author established  the inviscid limit  result for two dimensional compressible viscoelastic equations in the half plane in \cite{WX21}.  Similar conclusion was proved for two dimensional compressible non-resistive magnetohydrodynamics equations in \cite{CLX21}. These two results reveal a different phenomenon that both non-degeneracy deformation tensor and transverse background magnetic field can prevent the strong boundary layer formation. It is noted that the magnetic diffusion term is included in (\ref{1.1}) compared with \cite{CLX21}.

The main task of this paper is to prove that the solution to viscous MHD equations \eqref{1.1}-\eqref{1.4} converge to the solutions to the following ideal MHD equations as the small parameter $\varepsilon $ goes to zero.
\begin{align} \label{1.5}
\begin{cases}
\partial_t\rho^0+\nabla\cdot(\rho^0 {\bf{v}}^0)=0,\\
\partial_t(\rho^0 {\bf{v}}^0)+\mathrm{div}(\rho^0 {\bf{v}}^0\otimes {\bf{v}}^0) +\nabla p^0=(\nabla\times {\bf{B}}^0)\times {\bf{B}}^0,\\
\partial_t{\bf{B}}^0-\nabla\times({\bf{v}}^0\times {\bf{B}}^0)=0,\quad \mathrm{div}\ {\bf{B}}^0=0,
\end{cases}
\end{align}
where ${\bf{v}}^0=(v_1^0, v_2^0)$ is the velocity and ${\bf{B}}^0=(b_1^0, b_2^0)$ denotes the magnetic field.

 To formulate the main results, we  introduce the following conormal derivative operators of functions depending on $(t,\bf{x})$:
\begin{align*}
\mathcal{Z}_0=\partial_t,\quad \mathcal{Z}_1=\partial_x,\quad \mathcal{Z}_2=\phi(y)\partial_y, \quad \mathcal{Z}^\alpha=\mathcal{Z}_0^{\alpha_0}\mathcal{Z}_1^{\alpha_1}\mathcal{Z}_2^{\alpha_2},
\end{align*}
where  the spatial variables ${\bf{x}}=(x,  y)$, the multi-index $\alpha=(\alpha_0, \alpha_1, \alpha_2)$ and  $|\alpha|=\alpha_0+\alpha_1+\alpha_2$.  The smooth and bounded function  $\phi(y)$   satisfying $\phi(y)|_{y=0}=0$ and $\phi'|_{y=0}>0$, typically,  we can  choose $\phi(y)=\frac{y}{1+y}$.
	
 For any integer $m\in \mathbb{N}$, we denote the conormal Sobolev space
 \begin{align*}
 H_{co}^m([0, T] \times \mathbb{R}_+^2)=\{f(t, {\bf{x}}): \mathcal{Z}^\alpha f\in L^2([0, T]\times \mathbb{R}_+^2), \quad |\alpha|\leq m\}.
 \end{align*}
For any $t\geq 0$, we set the norms
 \begin{align*}
 \|f(t)\|_{m}^2=\sum_{|\alpha|\leq m}\|\mathcal{Z}^\alpha f(t,\cdot)\|_{L_x^2L_y^2}^2,
 \end{align*}
and
 \begin{align*}
 \|f\|_{H_{co}^m}^2=\int_0^t\|f(s)\|_{m}^2 ds.
 \end{align*}

As usual we use the notation
 \begin{align*}
 W_{co}^{m, \infty}([0, T] \times \mathbb{R}_+^2)=\{f(t, {\bf{x}}): \mathcal{Z}^\alpha f\in L^\infty([0, T]\times \mathbb{R}_+^2), \quad |\alpha|\leq m\},
 \end{align*}
 and
 \begin{align*}
 \|f\|_{m, \infty}=\sum_{|\alpha|\leq m}\|\mathcal{Z}^\alpha f \|_{L_{t, {\bf{x}}}^\infty}.
 \end{align*}

 It is also convenient to introduce the functional setting
 \begin{align*}
 \Lambda^m(t)=\{(\rho, {\bf{v}}, {\bf B}): \partial_y^i(\rho-1, {\bf{v}}, {\bf B}-\overset{\rightarrow}{e_y})\in H_{co}^{m-i},  i=0, 1 \}
 \end{align*}
 with $\overset{\rightarrow}{e_y}=(0, 1)$.

To derive the uniform conormal estimates of the classical solution  $(\rho^\varepsilon, {\bf{v}}^\varepsilon, {\bf{B}}^\varepsilon)$ to compressible MHD equations \eqref{1.1}-\eqref{1.4},  we introduce the  following energy functional:
\begin{align*}
N_m(t)=&\sum_{|\alpha|\leq m}\sup_{0\leq s\leq t}\int_{\mathbb{R}_+^2}\left(\rho^\varepsilon(s)|\mathcal{Z}^\alpha {\bf{v}}^\varepsilon(s)|^2+|\mathcal{Z}^\alpha ({\bf{B}}^\varepsilon-\overset{\rightarrow}{e_y})(s)|^2+\gamma^{-1}(p^\varepsilon)^{-1}|\mathcal{Z}^\alpha (p^\varepsilon(s)-1)|^2\right) d{\bf{x}} \\
&+ \|\partial_y({\bf{v}}^\varepsilon, b_1^\varepsilon, p^\varepsilon) \|_{H_{co}^{m-1}}^2
+\|\partial_y^2 v_2^\varepsilon \|_{H_{co}^{m-2}}^2+\varepsilon\mu \|\nabla {\bf{v}}^\varepsilon\|_{H_{co}^m}^2   +\varepsilon(\mu+\lambda) \|\nabla\cdot {\bf{v}}^\varepsilon\|_{H_{co}^m}^2\\
&+\varepsilon\kappa \|\nabla  {\bf{B}}^\varepsilon\|_{H_{co}^m}^2+\varepsilon^2\mu^2\kappa\|\partial_y^2 v_1^\varepsilon\|_{H_{co}^{m-1}}^2 +\varepsilon^2(2\mu+\lambda)^2 \|\partial_y^2 v_2^\varepsilon\|_{H_{co}^{m-1}}^2  +\varepsilon^2\kappa^2\mu \|\partial_y^2 b_1^\varepsilon\|_{H_{co}^{m-1}}^2.
\end{align*}
		
Now, it is position to state the main results of this paper.
\begin{theorem}\label{Th1}
(Uniform regularity estimates and inviscid limit) Let  the integer $m\geq 9$. Suppose the initial data $ (\rho^\varepsilon_0, \bf{v}^\varepsilon_0, \bf{B}^\varepsilon_0)$ satisfies
\begin{align}\label{1.6}
\sum_{i=0}^1\|\partial_y^i(p_0^\varepsilon-1, {\bf{v}}_0^\varepsilon, {\bf{B}}_0^\varepsilon-\overset{\rightarrow}{e_y})\|_{m-i}^2 \leq\sigma,
\end{align}
where $\sigma>0$ is some sufficiently small  constant. Then for the classical solution $U^\varepsilon=(\rho^\varepsilon, {\bf{v}}^\varepsilon, {\bf{B}}^\varepsilon)\in \Lambda^m(T)$ to the initial boundary value problem of viscous  compressible MHD equations \eqref{1.1}-\eqref{1.4}, there exists  a time $T>0$  independent  of $\varepsilon$, such that for any   $t\in [0, T]$,  the following regularity estimate holds:
\begin{align}
\label{1.7}
N_m(t)+ &\gamma^{-1}\varepsilon(2\mu+\lambda)\sum_{|\alpha|+i\leq m\atop i=1,2}\int_{\mathbb{R}_+^2}
(p^\varepsilon)^{-1}(t)|\mathcal{Z}^\alpha \partial_y^i p^\varepsilon(t)|^2  d{\bf{x}}
\le C\sigma
\end{align}
where $C>0$ is  some   constant, which is independent of $\varepsilon$.
	 	
Moreover, there exists a  unique solution   $U^0=(\rho^0, {\bf{v}}^0, {\bf{B}}^0)\in \Lambda^m(T)$  to the ideal compressible MHD equations \eqref{1.5}, such that
\begin{align*}
 \lim_{\varepsilon\rightarrow 0}\sup_{t\in[0, T]}\|(U^{\varepsilon}-U^0)(t, \cdot)\|_{L^\infty(\mathbb{R}_+^2)}=0.
\end{align*}
\end{theorem}
 Before proceeding, let us explain the difficulty and related strategy for the proof of main theorem. Compared to the previous work \cite{CLX21}, the presence of magnetic diffusion term $\varepsilon\kappa \Delta {\bf{B}}^\varepsilon$ in \eqref{1.1} will produce more mixed terms of higher-order derivative when we derive the conormal estimates of $\partial_y v_1$ and $\partial_y b_1$, which makes the analysis in this paper different from the arguments in \cite{CLX21}. Moreover, the conormal estimates for $\partial_y v_1$ and $\partial_y b_1$ should be estimated together by using the equations of $b_1$ and $v_1$ here. And the mixed terms appearing in the left hand sides of (\ref{4.2}) and (\ref{4.6}) are cancelled by using the boundary condition of $\partial_yb_1|_{y=0}=0$.

 Here, it should be emphasized that the uniform estimates of the first order normal derivative of $\partial_y({\bf{v}}^\varepsilon, b_1^\varepsilon, p^\varepsilon)$ and only the second order normal derivative of $\partial_y^2 v_2^\varepsilon$ are achieved in (\ref{1.7}). However, both the first order normal derivative of $\partial_y({\bf{v}}^\varepsilon, {\bf{B}}^\varepsilon, p^\varepsilon)$ and the second order normal derivative of $\partial_y^2({\bf{v}}^\varepsilon, {\bf{B}}^\varepsilon, p^\varepsilon)$ were obtained in \cite{CLX21}. In fact, we believe that the uniform estimates of high order normal derivative of $\partial_y^i({\bf{v}}^\varepsilon, {\bf{B}}^\varepsilon, p^\varepsilon)\ (i\geq 3)$ still can be derived in \cite{CLX21}, where there is no magnetic diffusion term of $\varepsilon\kappa \Delta {\bf{B}}^\varepsilon$ in \eqref{1.1}, provided that the high order compatibility conditions are assumed there. But, if there exists a diffusion term of $\varepsilon\kappa \Delta {\bf{B}}^\varepsilon$ in \eqref{1.1}, it seems that it is impossible to derive the uniform estimates of the high order normal derivatives, even for the second order normal derivative of $\partial_y^2(v_1^\varepsilon, b_1^\varepsilon, p^\varepsilon)$. As a consequence, we can only use the uniform estimates of  $\|({\bf{v}}^\varepsilon, {\bf{B}}^\varepsilon, p^\varepsilon, \partial_yv_2)\|_{L^\infty}$ to close the a priori energy estimates in this paper, which is truly different from the analysis in \cite{CLX21}. In other words, when there is  a diffusion term of $\varepsilon\kappa \Delta {\bf{B}}^\varepsilon$ in \eqref{1.1}, at this moment we only prove that the $O(1)$ order boundary layer does not appear when the vanishing viscosity of solution to (\ref{1.1})-(\ref{1.4}) is considered.

 In addition, the form of boundary condition \eqref{1.4} is not essential. Precisely, the value $1$ of $b^\varepsilon_2$ on the boundary can be replaced by any given function $f(t,x)$, which satisfies $0< c\leq f(t,x)\leq C$. Then, the results in Theorem \ref{Th1} are still believed to hold true. It is noted that all of these results only depend on the main assumption that the background magnetic field is transverse to the boundary. Finally, In addition to its own significance, this paper can be regarded as a complement of \cite{CLX21}. And both of these two results show that the transverse background magnetic field can prevent the strong boundary layer from formation whether there is magnetic diffusion term of $\varepsilon\kappa \Delta {\bf{B}}^\varepsilon$ in \eqref{1.1} or not.

The paper is organized as follows. In Section 2, we give some elementary lemmas. Section 3 is devoted to  the uniform conormal energy estimates of the classical solution to \eqref{1.1}-\eqref{1.4}.  We establish the conormal estimates for normal derivatives of the classical solution in Section 4. In Section 5, we prove Theorem \ref{Th1} based on the estimates obtained in Section 3 and Section 4.	

In the following parts,  we use  notation $A\lesssim B$  to present $A\leq CB$ for some generic constant $C>0$ independent of $\varepsilon$.
And we denote the polynomial functions by $\mathcal{P}(\cdot)$, which may vary from line to line.  The commutator is expressed by $[\cdot, \cdot]$.

\section{Preliminaries}
In this section,  we present some elementary lemmas that will be used frequently later. The first one is the Sobolev-Gagliardo-Nirenberg-Moser type inequality for the conormal Sobolev space and its proof can be found in \cite{G90}.
\begin{lemma}\label{lem 2.2}
For the functions $f, g \in L^\infty([0,T]\times\mathbb{R}_+^2)\cap H_{co}^m([0,T]\times\mathbb{R}_+^2)$ with $m\in \mathbb{N}$, it holds that for any  $\alpha, \beta\in \mathbb{N}^3$ with $|\alpha|+|\beta|=m$,
\begin{align}\label{2.1}
\int_0^t\|(\mathcal{Z}^\alpha f\mathcal{Z}^\beta g) (s)\|^2 ds\lesssim \|f\|_{L_{t,\bf{x}}^\infty}^2\int_0^t\|g(s)\|_{m}^2 ds+\|g\|_{L_{t,\bf{x}}^\infty}^2\int_0^t\|f(s)\|_{m}^2 ds.
\end{align}
\end{lemma}
Next, the anisotropic Sobolev embedding property in the conormal  Sobolev space can be found in \cite{P16}.
\begin{lemma}\label{L2}
Suppose that  $f(t,{\bf{x}})\in H_{co}^3([0, t]\times \mathbb{R}^2_+)$ and $\partial_y f\in H_{co}^2([0, t]\times \mathbb{R}^2_+)$, then
\begin{align*}
\|f\|_{L_{t, {\bf{x}}}^\infty}^2\lesssim\|f(0)\|_{2}^2+\|\partial_yf(0)\|_{1}^2+\int_0^t\left(\|f(s)\|_{3}^2+\|\partial_yf(s)\|_{2}^2\right)ds.
\end{align*}
\end{lemma}
To deal with the commutator involving conormal derivatives, we need the following properties of commutators.  The proof can be found in \cite{P16}.	For any integer $m\geq 1$, there exist two families of bounded smooth functions $\{\phi_{k,m}(y)\}_{0\leq k\leq m-1}$ and $\{\phi^{k,m}(y)\}_{0\leq k\leq m-1}$ depending only on $\phi(y)$, such that
\begin{align}\label{2.2}
[\mathcal{Z}_2^m, \partial_y]=\sum_{k=0}^{m-1}\phi_{k,m}(y)\mathcal{Z}_y^k\partial_y=\sum_{k=0}^{m-1}\phi^{k,m}(y)\partial_y\mathcal{Z}_y^k.
\end{align}

\section{Conormal Energy Estimates}
For simplicity, we omit the superscript $\varepsilon$ in the  rest of paper without causing confusion.	
The conormal energy estimates of the classical solution $(\rho , {\bf{v}}, {\bf{B}} )$ are considered in this section. We rewrite the system \eqref{1.1} as follows.
\begin{align}\label{3.1}
\begin{cases}
\partial_t \rho+{\bf{v}}\cdot\nabla\rho +\rho\nabla\cdot {\bf{v}}=0,\\
\rho\partial_t {\bf{v}}+\rho {\bf{v}}\cdot\nabla {\bf{v}}-\varepsilon \mu \Delta  {\bf{v}}-\varepsilon(\mu+\lambda)\nabla(\nabla\cdot {\bf{v}})+ \nabla p =(\nabla\times {\bf{B}})\times {\bf{B}},\\
\partial_t{\bf{B}}-\nabla\times({\bf{v}}\times {\bf{B}}) =\varepsilon\kappa \Delta {\bf{B}}^\varepsilon, \quad \mathrm{div}\ {\bf{B}}=0.
\end{cases}
\end{align}
\begin{lemma}\label{lem1}
Under the assumption in Theorem \ref{Th1}, the classical solution $(\rho, {\bf{v}}, {\bf{B}})$  to the viscous compressible MHD equations \eqref{3.1} with the  boundary conditions \eqref{1.3}-\eqref{1.4} satisfies
\begin{align*}
&\sum\limits_{|\alpha|\leq m}\int_{\mathbb{R}_+^2}\left(\rho(t)|\mathcal{Z}^\alpha {\bf{v}}(t)|^2+|\mathcal{Z}^\alpha ({\bf{B}}(t)-\overset{\rightarrow}{e_y})|^2  +  \gamma^{-1} p^{-1}(t)|\mathcal{Z}^\alpha (p(t)-1)|^2\right) d{\bf{x}}\\
&+\varepsilon\mu\int_0^t\|\nabla {\bf{v}}\|_{m}^2\ ds+\varepsilon(\mu+\lambda)\int_0^t \|\nabla \cdot {\bf{v}}\|_{m}^2 ds+\varepsilon\kappa\int_0^t \| \nabla {\bf{B}} \|_{m}^2\ ds\\
\lesssim & \sum\limits_{|\alpha|\leq m}\int_{\mathbb{R}_+^2}\left(\rho_0|\mathcal{Z}^\alpha {\bf{v}}_0|^2+|\mathcal{Z}^\alpha ({\bf{B}}_0-\overset{\rightarrow}{e_y})|^2  +  \gamma^{-1} p_0^{-1}|\mathcal{Z}^\alpha (p_0-1)|^2\right) d{\bf{x}}\\
&+\left[ \left(1+\|(p, {\bf{v}}, {\bf{B}}-\overset{\rightarrow}{e_y}, \partial_yv_2)\|_{[ m/2]+1, \infty}^2\right)^3+
\| (\partial_y {\bf{v}}, \partial_y {\bf{B}})(0)\|_{ [m/2] +2}^2\right.\\
&+\left.  \int_0^t\|(\partial_y {\bf{v}}, \partial_y {\bf{B}})\|_{[m/2]+2}^2\ ds  +\varepsilon\|(v_1, \partial_yv_1, \partial_yb_1)\|_{2, \infty}^2 \left(1+\|(\rho, p^{-1}, b_1, \partial_y v_2)\|_{L_{t, {\bf{x}}}^\infty}^2 \right) ^2 \right] \notag\\
&\cdot\sum_{j=0}^1 \int_0^t \|\partial_y^j({\bf{v}}, {\bf{B}}-\overset{\rightarrow}{e_y}, p-1) \|_{m-j}^2 \ ds.
\end{align*}
\end{lemma}
\begin{proof}
For any multi-index $\alpha$ satisfying $|\alpha|\leq m$,  we apply the conormal derivatives $\mathcal{Z}^\alpha$ to the last two equations in \eqref{3.1}. By multiplying $(\mathcal{Z}^\alpha {\bf{v}}, \mathcal{Z}^\alpha ({\bf{B}}-\overset{\rightarrow}{e_y}))$ on both sides of the resulting equalities and integrating them over $[0, t]\times\mathbb{R}_+^2$, we get
\begin{align}\label{3.2}
&\frac12	 \int_{\mathbb{R}_+^2}\left(\rho(t)|\mathcal{Z}^\alpha {\bf{v}}(t)|^2  +	 |\mathcal{Z}^\alpha ({\bf{B}}(t)-\overset{\rightarrow}{e_y})|^2 \right)d{\bf{x}}-\frac12	 \int_{\mathbb{R}_+^2}\left(\rho_0|\mathcal{Z}^\alpha {\bf{v}}_0|^2  +	 |\mathcal{Z}^\alpha ({\bf{B}}_0-\overset{\rightarrow}{e_y})|^2 \right)d{\bf{x}}  \nonumber\\
=&\varepsilon \mu\int_0^t\int_{\mathbb{R}_+^2}  \mathcal{Z}^\alpha \Delta {\bf{v}} \cdot \mathcal{Z}^\alpha {\bf{v}} \ d{\bf{x}}ds+\varepsilon(\mu+\lambda)\int_0^t\int_{\mathbb{R}_+^2} \mathcal{Z}^\alpha\nabla(\nabla\cdot {\bf{v}})\cdot \mathcal{Z}^\alpha {\bf{v}} \ d{\bf{x}}ds \nonumber\\
&+\int_0^t\int_{\mathbb{R}_+^2}  \mathcal{C}_1^\alpha \cdot \mathcal{Z}^\alpha {\bf{v}}\  d{\bf{x}}ds+\int_0^t\int_{\mathbb{R}_+^2}  \mathcal{C}_2^\alpha \cdot \mathcal{Z}^\alpha {\bf{v}}\  d{\bf{x}}ds+\varepsilon\kappa\int_0^t\int_{\mathbb{R}_+^2}\mathcal{Z}^\alpha\Delta {\bf{B}}^\varepsilon\cdot\mathcal{Z}^\alpha ({\bf{B}}-\overset{\rightarrow}{e_y})\  d{\bf{x}}ds \nonumber\\
&-\int_0^t\int_{\mathbb{R}_+^2} \mathcal{Z}^\alpha\nabla p\cdot \mathcal{Z}^\alpha {\bf{v}} \ d{\bf{x}}ds   +\int_0^t\int_{\mathbb{R}_+^2} \mathcal{Z}^\alpha [(\nabla\times {\bf{B}})\times {\bf{B}}]\cdot \mathcal{Z}^\alpha {\bf{v}} \ d{\bf{x}}ds\nonumber\\ & +\int_0^t\int_{\mathbb{R}_+^2} \mathcal{Z}^\alpha[\nabla\times({\bf{v}}\times {\bf{B}})]\cdot\mathcal{Z}^\alpha ({\bf{B}}-\overset{\rightarrow}{e_y})\  d{\bf{x}}ds,
\end{align}
where
\begin{align*}
&\mathcal{C}_1^\alpha=-[Z^\alpha, \rho\partial_t]{\bf v} =-\sum_{\beta+\gamma=\alpha\atop |\beta|\geq 1} C_\alpha^\beta \mathcal{Z}^\beta \rho\mathcal{Z}^\gamma\partial_t {\bf{v}},
\end{align*}
and 	
\begin{align*}
\mathcal{C}_2^\alpha =-[Z^\alpha, \rho{\bf v}\cdot\nabla]{\bf v}=-\sum_{\beta+\gamma=\alpha\atop |\beta|\geq 1}C_\alpha^\beta \mathcal{Z}^\beta(\rho {\bf{v}})\cdot\mathcal{Z}^\gamma \nabla {\bf{v}}-\rho v_2\cdot[\mathcal{Z}^\alpha, \partial_y]{\bf{v}}.
\end{align*}
 In what follows, we deal with term by term of  \eqref{3.2}. The first three terms on the right hand side of \eqref{3.2} have the same estimates as \cite{CLX21}.

 To consider the fourth  term on the right hand side of \eqref{3.2}, we have
 \begin{align}\label{change 4}
 &\int_0^t\int_{\mathbb{R}_+^2}  \mathcal{C}_2^\alpha \cdot \mathcal{Z}^\alpha {\bf{v}} \ d{\bf{x}}ds=\sum_{\beta+\gamma=\alpha\atop |\beta|\geq 1}C_\alpha^\beta\int_0^t\int_{\mathbb{R}_+^2} \Big(\mathcal{Z}^\beta(\rho v_1)\cdot\mathcal{Z}^\gamma \partial_x{\bf{v}}\cdot\mathcal{Z}^\alpha {\bf{v}}-\rho  v_2[\mathcal{Z}^\alpha, \partial_y]{\bf{v}}\cdot\mathcal{Z}^\alpha {\bf{v}}\notag\\
 &+ \mathcal{Z}^\beta(\rho v_2)\cdot\mathcal{Z}^\gamma \partial_y v_2\cdot\mathcal{Z}^\alpha v_2\Big)\ d{\bf{x}}ds+\sum_{\beta+\gamma=\alpha\atop |\beta|\geq 1}C_\alpha^\beta\int_0^t\int_{\mathbb{R}_+^2}  \mathcal{Z}^\beta(\rho v_2)\cdot\mathcal{Z}^\gamma \partial_yv_1\cdot\mathcal{Z}^\alpha v_1  \ d{\bf{x}}ds.
 \end{align}
The first term on the right hand side of \eqref{change 4} is estimated by
 \begin{align*}
 &\sum_{\beta+\gamma=\alpha\atop |\beta|\geq 1}C_\alpha^\beta\int_0^t\int_{\mathbb{R}_+^2} \Big(\mathcal{Z}^\beta(\rho v_1)\cdot\mathcal{Z}^\gamma \partial_x{\bf{v}}\cdot\mathcal{Z}^\alpha {\bf{v}}-\rho  v_2[\mathcal{Z}^\alpha, \partial_y]{\bf{v}}\cdot\mathcal{Z}^\alpha {\bf{v}}\notag\\
 &+ \mathcal{Z}^\beta(\rho v_2)\cdot\mathcal{Z}^\gamma \partial_y v_2\cdot\mathcal{Z}^\alpha v_2\Big) \ d{\bf{x}}ds\notag\\
 \lesssim& \|(\rho, {\bf{v}}, \partial_y v_2)\|_{ 1,\infty}^2\sum_{j=0}^1\left( \int_0^t\|\partial_y^j(\rho-1, {\bf{v}})\|_{m-j}^2 \; ds \right)^\frac12\left(\int_0^t\|{\bf{v}}\|_{m}^2  \;ds\right)^\frac12.
 \end{align*}
 For the second term on the right hand side of \eqref{change 4}, we have
 \begin{align}\label{change 5}
 &\sum_{\beta+\gamma=\alpha\atop |\beta|\geq 1}C_\alpha^\beta\int_0^t\int_{\mathbb{R}_+^2}  \mathcal{Z}^\beta(\rho v_2)\cdot\mathcal{Z}^\gamma \partial_y v_1\cdot\mathcal{Z}^\alpha v_1 \ d{\bf{x}}ds \notag\\
 =&\sum_{\beta+\gamma=\alpha \atop 1\leq|\beta|\leq |\gamma|}C_\alpha^\beta\int_0^t\int_{\mathbb{R}_+^2}   \mathcal{Z}^\beta(\rho v_2)\cdot\mathcal{Z}^\gamma \partial_y v_1\cdot\mathcal{Z}^\alpha v_1 \ d{\bf{x}}ds\notag\\
 &+\sum_{\beta+\gamma=\alpha\atop 1\leq |\gamma|<|\beta|}C_\alpha^\beta\int_0^t\int_{\mathbb{R}_+^2}   \mathcal{Z}^\beta(\rho v_2)\cdot\mathcal{Z}^\gamma \partial_y v_1\cdot\mathcal{Z}^\alpha v_1\  d{\bf{x}}ds\notag\\
  &+ \int_0^t\int_{\mathbb{R}_+^2}  \mathcal{Z}^\alpha(\rho v_2)\cdot  \partial_y v_1\cdot\mathcal{Z}^\alpha v_1 \ d{\bf{x}}ds.
 \end{align}
 It is direct to  bound the first two terms on the right hand side of \eqref{change 5} by
 \begin{align*}
 &\|(\rho, v_2)\|_{[m/2], \infty}\left(\int_0^t\|\partial_y v_1\|_{m-1}^2\; ds\right)^\frac12\left(\int_0^t\| v_1\|_{m}^2\; ds\right)^\frac12\notag\\
 &+\|(\rho, \phi \partial_y v_1, \phi^{-1} v_2)\|_{[m/2], \infty}
 \left(\int_0^t\|(\rho-1, \phi^{-1}v_2)\|_{m-1}^2\; ds\right)^\frac12\left(\int_0^t\|v_1\|_{m}^2\; ds\right)^\frac12\notag\\
  \lesssim&\left(1+\|(\rho,   {\bf{v}}, \partial_y  v_2)\|_{[m/2]+1, \infty}^2\right)
 \left(\int_0^t\|(\rho-1, \partial_y{\bf{v}})\|_{m-1}^2\; ds+\int_0^t\|v_1\|_{m}^2\; ds\right).
 \end{align*}

 Next, we consider the third term on the right hand side of \eqref{change 5}.
 Let $\tilde b_2=b_2-1$. Then  the equation of $b_1$  can be rewritten as follows.
 \begin{align}\label{4.1}
 \partial_yv_1+\varepsilon\kappa\partial_y^2 b_1=-\varepsilon\kappa\partial_x^2 b_1+\partial_t b_1+v_1 \partial_xb_1-  \tilde b_2\partial_yv_1 +v_2\partial_y b_1+b_1\partial_y v_2.
 \end{align}	

 Inserting above equality into the third term on the right hand side of \eqref{change 5}, we have
 \begin{align}\label{change 7}
 & \int_0^t\int_{\mathbb{R}_+^2}  \mathcal{Z}^\alpha(\rho v_2)\cdot  \partial_y v_1\cdot\mathcal{Z}^\alpha v_1 \; d{\bf{x}}ds\notag\\
 =&-\varepsilon\kappa \int_0^t\int_{\mathbb{R}_+^2}   \mathcal{Z}^\alpha(\rho v_2)\cdot  \partial_y^2 b_1\cdot\mathcal{Z}^\alpha v_1 \;d{\bf{x}}ds \notag\\
 &+ \int_0^t\int_{\mathbb{R}_+^2}  \mathcal{Z}^\alpha(\rho v_2)\cdot \Big(-\varepsilon\kappa  \partial_x^2b_1+\partial_t  b_1+  v_1\partial_x b_1
 + b_1\partial_y v_2+ v_2\partial_y b_1- \widetilde b_2\partial_y v_1\Big)\cdot\mathcal{Z}^\alpha v_1\; d{\bf{x}}ds.
 \end{align}
By \eqref{2.2}, the first part on the right hand side of \eqref{change 7} can be handled in the following way.
\begin{align*}
 &-\varepsilon\kappa \int_0^t\int_{\mathbb{R}_+^2}   \mathcal{Z}^\alpha(\rho v_2)\cdot  \partial_y^2 b_1\cdot\mathcal{Z}^\alpha v_1 \;d{\bf{x}}ds \\
 =& -\varepsilon\kappa\sum_{\beta+\gamma=\alpha} \int_0^t\int_{\mathbb{R}_+^2}   \mathcal{Z}^\beta\rho\phi^{-1} \mathcal{Z}^\gamma v_2\cdot \phi \partial_y^2 b_1\cdot\mathcal{Z}^\alpha v_1 \;d{\bf{x}}ds\\
 =& -\varepsilon\kappa\sum_{\beta+\gamma=\alpha} \int_0^t\int_{\mathbb{R}_+^2}   \mathcal{Z}^\beta\rho\phi^{-1} \mathcal{Z}^\gamma v_2\cdot \mathcal{Z}_2\partial_y  b_1\cdot\mathcal{Z}^\alpha v_1 \;d{\bf{x}}ds \\
 \lesssim&\varepsilon\|\mathcal{Z}_2\partial_y  b_1\|_{L_{t, {\bf{x}}}^\infty}\|(\rho, \partial_y v_2)\|_{L_{t, {\bf{x}}}^\infty}\left(\int_0^t\|(\rho-1,  \partial_y v_2)\|_{m}^2\; ds\right)^\frac12\left(\int_0^t\| v_1\|_{m}^2 \;ds\right)^\frac12\\
 \lesssim& \varepsilon\|\partial_yb_1\|_{1, \infty}^2\|(\rho, \partial_y v_2)\|_{L_{t, {\bf{x}}}^\infty}^2\int_0^t\|(p-1, v_1)\|_{m}^2\; ds+\frac{\varepsilon\mu^2}6\int_0^t\|  \partial_yv_2\|_{m}^2 \;ds.
 \end{align*}The remaining terms on the right hand side of \eqref{change 7} can be estimated directly.
 \begin{align*}
 & \int_0^t\int_{\mathbb{R}_+^2}  \mathcal{Z}^\alpha(\rho v_2)\cdot \Big(-\varepsilon\kappa  \partial_x^2b_1+\partial_t  b_1+  v_1\partial_x b_1
 + b_1\partial_y v_2+ v_2\partial_y b_1- \widetilde b_2\partial_y v_1\Big)\cdot\mathcal{Z}^\alpha v_1 \;d{\bf{x}}ds \notag\\
 \lesssim&\Big(1+\|(\rho, v_1, \phi \partial_y v_1, v_2, \phi^{-1}v_2, \partial_y v_2, b_1,   \phi\partial_y b_1, \phi^{-1}\widetilde b_2 )\|_{2,\infty}\Big)^3\left(\int_0^t\|(\rho-1, v_2)\|_{m}^2\; ds\right)^\frac12\notag\\
 &\cdot\left(\int_0^t\|  v_1\|_{m}^2 \;ds\right)^\frac12\notag\\
 \lesssim&\Big(1+\|(\rho, v_1,  v_2,  b_1, \partial_yv_2 )\|_{3,\infty}\Big)^3 \int_0^t\|(\rho-1, {\bf{v}})\|_{m}^2\; ds.
 \end{align*}

 To handle the fifth term on the right hand side of \eqref{3.2},  by integration by parts, we have
\begin{align}\label{3.3}
&\varepsilon\kappa\int_0^t\int_{\mathbb{R}_+^2}\mathcal{Z}^\alpha\Delta {\bf{B}} \cdot\mathcal{Z}^\alpha ({\bf{B}}-\overset{\rightarrow}{e_y})\  d{\bf{x}}ds  \notag\\
= &-\varepsilon\kappa\int_0^t \|\nabla {\bf{B}} \|_{m}^2 ds+\varepsilon \kappa\int_0^t\int_{\mathbb{R}_+^2}[\mathcal{Z}^\alpha, \nabla\cdot]\nabla  {\bf{B}} \cdot\mathcal{Z}^\alpha ({\bf{B}}-\overset{\rightarrow}{e_y})\  d{\bf{x}}ds \notag\\
&+\varepsilon \kappa\int_0^t\int_{\mathbb{R}_+^2}\mathcal{Z}^\alpha\nabla  {\bf{B}} \cdot[\mathcal{Z}^\alpha, \nabla] ({\bf{B}}-\overset{\rightarrow}{e_y})\ d{\bf{x}}ds.
\end{align}
For the  second term on the right hand side of \eqref{3.3}, 	due to \eqref{2.2},  one has
\begin{align}\label{3.4}
&\varepsilon \kappa\int_0^t\int_{\mathbb{R}_+^2}[\mathcal{Z}^\alpha, \nabla\cdot]\nabla  {\bf{B}} \cdot\mathcal{Z}^\alpha ({\bf{B}}-\overset{\rightarrow}{e_y})\  d{\bf{x}}ds \notag\\
=&\varepsilon \kappa\sum_{k=0}^{m-1}\int_0^t\int_{\mathbb{R}_+^2} \phi^{k, m}(y)\partial_y\mathcal{Z}_y^k\partial_y b_1 \cdot\mathcal{Z}^\alpha  b_1\  d{\bf{x}}ds \notag\\
&+\varepsilon \kappa\sum_{k=0}^{m-1}\int_0^t\int_{\mathbb{R}_+^2} \phi^{k, m}(y)\partial_y\mathcal{Z}_y^k\partial_y b_2 \cdot\mathcal{Z}^\alpha  (b_2-1)\  d{\bf{x}}ds.
\end{align}
By integration by parts,  the first term on the right hand side of \eqref{3.4} is estimated as follows.
\begin{align*}
&\varepsilon \kappa\sum_{k=0}^{m-1}\int_0^t\int_{\mathbb{R}_+^2} \phi^{k, m}(y)\partial_y\mathcal{Z}_y^k\partial_y b_1 \cdot\mathcal{Z}^\alpha  b_1\  d{\bf{x}}ds \\
=&-\varepsilon \kappa\sum_{k=0}^{m-1}\int_0^t\int_{\mathbb{R}_+^2} \partial_y\phi^{k, m}(y)\mathcal{Z}_y^k\partial_y b_1 \cdot\mathcal{Z}^\alpha  b_1\  d{\bf{x}}ds \\
&-\varepsilon \kappa\sum_{k=0}^{m-1}\int_0^t\int_{\mathbb{R}_+^2} \phi^{k, m}(y)\mathcal{Z}_y^k\partial_y b_1 \cdot\partial_y\mathcal{Z}^\alpha  b_1\  d{\bf{x}}ds \\
\leq& \frac{\varepsilon\kappa^2}{6}\int_0^t\|\partial_y b_1\|_{m}^2\; ds+ \int_0^t\left(\| b_1  \|_{m}^2+\|\partial_yb_1  \|_{m-1}^2\right)\; ds.
\end{align*}
Since $\nabla\cdot {\bf{B}}=0$, the  second term on the right hand side of \eqref{3.4} satisfies
\begin{align*}
&\varepsilon \kappa\sum_{k=0}^{m-1}\int_0^t\int_{\mathbb{R}_+^2} \phi^{k, m}(y)\partial_y\mathcal{Z}_y^k\partial_y b_2 \cdot\mathcal{Z}^\alpha   (b_2-1)\  d{\bf{x}}ds\notag\\
=&-\varepsilon \kappa\sum_{k=0}^{m-1}\int_0^t\int_{\mathbb{R}_+^2} \phi^{k, m}(y)\partial_y\mathcal{Z}_y^k\partial_x b_1 \cdot\mathcal{Z}^\alpha (b_2-1)\  d{\bf{x}}ds\notag\\
\leq&\frac{\varepsilon\kappa^2}{6}\int_0^t\|\partial_y b_1\|_{m}^2 \;ds +C \int_0^t\|b_2-1\|_{m}^2 \;ds.
\end{align*}
For the third term on the right hand side of \eqref{3.3}, by \eqref{2.2},  we have
\begin{align*}
\varepsilon \kappa\int_0^t\int_{\mathbb{R}_+^2}\mathcal{Z}^\alpha\nabla  {\bf{B}} \cdot[\mathcal{Z}^\alpha, \nabla] ({\bf{B}}-\overset{\rightarrow}{e_y})\ d{\bf{x}}ds
\leq&\frac{\varepsilon \kappa^2}{6}\int_0^t\| \nabla{\bf{B}}\|_{m}^2\; ds+C\int_0^t\|\partial_y{\bf{B}}\|_{m-1}^2 \;ds.
\end{align*}
	
 Then  we write  the  sixth term on the right hand side of \eqref{3.2} as follows
\begin{equation}\label{change 1}
\begin{split}
& -\int_0^t\int_{\mathbb{R}_+^2} \mathcal{Z}^\alpha\nabla p\cdot \mathcal{Z}^\alpha {\bf{v}}\  d{\bf{x}}ds\\
=&\int_0^t\int_{\mathbb{R}_+^2} \mathcal{Z}^\alpha (p-1)\cdot [\nabla\cdot, \mathcal{Z}^\alpha] {\bf{v}} \ d{\bf{x}}d  - \int_0^t\int_{\mathbb{R}_+^2}[ \mathcal{Z}^\alpha,  \nabla] (p-1)\cdot \mathcal{Z}^\alpha {\bf{v}}\ d{\bf{x}}ds \\
&+\int_0^t\int_{\mathbb{R}_+^2}  \mathcal{Z}^\alpha (p-1)\cdot \mathcal{Z}^\alpha(\nabla\cdot {\bf{v}}) \ d{\bf{x}}ds.
\end{split}
\end{equation}
The first two terms have the same estimates as \cite{CLX21}. Hence, we only pay attention to the last term of  \eqref{change 1}.
 We   insert the equation of momentum
\begin{align*}
\nabla\cdot{\bf{v}}=-\gamma^{-1} p^{-1}\partial_t p-\gamma^{-1} p^{-1}{\bf{v}}\cdot \nabla p,
\end{align*}
  into  the  third term on the right hand side of \eqref{change 1} to get
\begin{align}\label{change 2}
&\int_0^t\int_{\mathbb{R}_+^2} \mathcal{Z}^\alpha (p-1)\cdot \mathcal{Z}^\alpha (\nabla\cdot {\bf{v}}) \ d{\bf{x}}ds\notag\\
= & -\gamma^{-1}\sum_{\beta+\gamma=\alpha } C_\alpha^\beta\int_0^t\int_{\mathbb{R}_+^2} \mathcal{Z}^\alpha (p-1)\cdot \mathcal{Z}^\beta   p^{-1}\partial_t\mathcal{Z}^\gamma   (p-1)\ d{\bf{x}}ds\notag\\
&-\gamma^{-1}\int_0^t\int_{\mathbb{R}_+^2} \mathcal{Z}^\alpha (p-1)\cdot   p^{-1}{\bf{v}}\cdot	\mathcal{Z}^\alpha  \nabla (p-1)\ d{\bf{x}}ds\notag\\
 &-\gamma^{-1}\sum_{\beta+  \gamma+\iota=\alpha\atop  |\iota|\neq |\alpha|} C_\alpha^\beta \int_0^t\int_{\mathbb{R}_+^2} \mathcal{Z}^\alpha (p-1) \cdot \mathcal{Z}^\beta   p^{-1}\mathcal{Z}^\gamma{\bf{v}}\cdot\mathcal{Z}^\iota \nabla (p-1)\ d{\bf{x}}ds.
\end{align}
 The first term shares the  same estimates as \cite{CLX21}.
Then we write the  second term on the right hand side of \eqref{change 2} as follows
\begin{align}\label{change 3}
&-\gamma^{-1}\int_0^t\int_{\mathbb{R}_+^2} \mathcal{Z}^\alpha (p-1)\cdot   p^{-1}{\bf{v}}\cdot	\mathcal{Z}^\alpha  \nabla (p-1)  \ d{\bf{x}}ds\notag\\
=&-\gamma^{-1}\int_0^t\int_{\mathbb{R}_+^2}  \mathcal{Z}^\alpha (p-1)\cdot   p^{-1}{\bf{v}}\cdot[\mathcal{Z}^\alpha , \nabla] (p-1)  \  d{\bf{x}}ds\notag\\
&-\gamma^{-1}\int_0^t\int_{\mathbb{R}_+^2} \mathcal{Z}^\alpha (p-1)\cdot  p^{-1}{\bf{v}}\cdot\nabla\mathcal{Z}^\alpha  (p-1)\ d{\bf{x}}ds.
\end{align}
 The first term on the right hand side of the above equation shares the same estimate as \cite{CLX21}. The second term on the right hand side of \eqref{change 3} can be treated as follows.
\begin{align*}
&\gamma^{-1}\int_0^t\int_{\mathbb{R}_+^2} \mathcal{Z}^\alpha (p-1)\cdot  p^{-1}{\bf{v}}\cdot\nabla\mathcal{Z}^\alpha  (p-1) \ d{\bf{x}}ds\\
=&-\frac{1}{2\gamma}\int_0^t\int_{\mathbb{R}_+^2}\nabla\cdot(  p^{-1}{\bf{v}})|\mathcal{Z}^\alpha (p-1)|^2 d{\bf{x}}ds\\
\lesssim &\left(\|\partial_x p^{-1}\|_{L_{t,{\bf{x}}}^\infty} \|v_1\|_{L_{t,{\bf{x}}}^\infty}+\|\phi\partial_y p^{-1}\|_{L_{t,{\bf{x}}}^\infty} \|\phi^{-1}v_2\|_{L_{t,{\bf{x}}}^\infty}+\| p^{-1}\|_{L_{t,{\bf{x}}}^\infty} \|\nabla\cdot {\bf{v}}\|_{L_{t,{\bf{x}}}^\infty}\right)\int_0^t\|p-1\|_{m}^2\; ds\notag\\
\lesssim &\left(\|  p^{-1}\|_{1,\infty} \|v_1\|_{1,\infty}+\| p^{-1}\|_{1,\infty} \|\partial_yv_2\|_{L_{t,{\bf{x}}}^\infty} \right)\int_0^t\|p-1\|_{m}^2\; ds.
\end{align*}
Next, we divide the last term  on the right hand side of \eqref{change 2}   into three parts.
\begin{align}\label{change 26}
 &-\gamma^{-1}\sum_{\beta+\gamma+\iota=\alpha\atop  |\iota|\neq |\alpha|} C_\alpha^\beta\int_0^t\int_{\mathbb{R}_+^2} \mathcal{Z}^\alpha (p-1) \cdot \mathcal{Z}^\beta  p^{-1}\mathcal{Z}^\gamma {\bf{v}}\cdot\mathcal{Z}^\iota \nabla (p-1)\ d{\bf{x}}ds\notag\\
 =&-\gamma^{-1}\sum_{\beta+\gamma+\iota=\alpha\atop |\iota|\neq |\alpha|} C_\alpha^\beta\int_0^t\int_{\mathbb{R}_+^2} \mathcal{Z}^\alpha (p-1) \cdot \mathcal{Z}^\beta  p^{-1}\mathcal{Z}^\gamma v_1\cdot\partial_x\mathcal{Z}^\iota (p-1)\ d{\bf{x}}ds\notag\\
 &-\gamma^{-1}\sum_{\beta+\gamma+\iota=\alpha\atop   |\gamma|, |\iota|\neq |\alpha|} C_\alpha^\beta\int_0^t\int_{\mathbb{R}_+^2} \mathcal{Z}^\alpha (p-1) \cdot \mathcal{Z}^\beta  p^{-1}\mathcal{Z}^\gamma v_2\cdot\mathcal{Z}^\iota  \partial_y(p-1)\ d{\bf{x}}ds\notag\\
 &-\gamma^{-1}   \int_0^t\int_{\mathbb{R}_+^2} \mathcal{Z}^\alpha (p-1) \cdot   p^{-1}\mathcal{Z}^\alpha v_2\cdot\partial_y(p-1)\ d{\bf{x}}ds.
\end{align}
For the first term on the right hand side of \eqref{change 26}, we have
\begin{align*}
&-\gamma^{-1}\sum_{\beta+\gamma+\iota=\alpha\atop |\iota|\neq |\alpha|} C_\alpha^\beta\int_0^t\int_{\mathbb{R}_+^2} \mathcal{Z}^\alpha (p-1) \cdot \mathcal{Z}^\beta   p^{-1}\mathcal{Z}^\gamma v_1\cdot\partial_x\mathcal{Z}^\iota   (p-1)\ d{\bf{x}}ds\notag\\
\lesssim&\|(v_1, p^{-1}, p)\|_{1, \infty}^2 \left(\int_0^t\|(v_1, p^{-1}-1, p-1)\|_{m}^2\; ds\right)^\frac12 \left(\int_0^t\|  p-1\|_{m}^2 \;ds\right)^\frac12.
\end{align*}
As for the second term on the right hand side of \eqref{change 26}, it holds
\begin{align}
&-\gamma^{-1}\sum_{\beta+\gamma+\iota=\alpha\atop   |\gamma|, |\iota|\neq |\alpha|} C_\alpha^\beta\int_0^t\int_{\mathbb{R}_+^2} \mathcal{Z}^\alpha (p-1) \cdot \mathcal{Z}^\beta   p^{-1}\mathcal{Z}^\gamma v_2\cdot\mathcal{Z}^\iota  \partial_y (p-1)\ d{\bf{x}}ds\notag\\
\lesssim&\Big[\sum_{\beta+\gamma+\iota=\alpha\atop |\gamma|, |\iota|<|\beta|}\| \phi^{-1} \mathcal{Z}^\gamma v_2\|_{L_{t, {\bf{x}}}^\infty}\|\phi \mathcal{Z}^\iota \partial_y(p-1)\|_{L_{t, {\bf{x}}}^\infty}\left(\int_0^t\|\mathcal{Z}^\beta p^{-1} \|_{L_x^2L_y^2} ^2\; ds\right)^\frac12\notag\\
&+\sum_{\beta+\gamma+\iota=\alpha\atop |\beta|, |\iota|\leq |\gamma|<|\alpha|} \|\mathcal{Z}^\beta p^{-1} \|_{L_{t, {\bf{x}}}^\infty}\|\phi \mathcal{Z}^\iota \partial_y(p-1) \|_{L_{t, {\bf{x}}}^\infty}\left(\int_0^t\|\phi^{-1} \mathcal{Z}^\gamma v_2\|_{L_x^2L_y^2} ^2 \; ds\right)^\frac12\notag\\
&+\sum_{\beta+\gamma+\iota=\alpha\atop |\beta|, |\gamma|\leq |\iota|<|\alpha|} \|\mathcal{Z}^\beta p^{-1} \|_{L_{t, {\bf{x}}}^\infty}\|\phi^{-1} \mathcal{Z}^\gamma v_2 \|_{L_{t, {\bf{x}}}^\infty}\left(\int_0^t\|\phi \mathcal{Z}^\iota\partial_y (p-1)\|_{L_x^2L_y^2} ^2 \; ds\right)^\frac12\Big]\notag\\ &\cdot\left(\int_0^t\|\mathcal{Z}^\alpha(p-1)\|_{L_x^2L_y^2} ^2  \;ds\right)^\frac12\notag\\
\lesssim&\|(p, p^{-1}, \partial_y v_2)\|_{[m/2]+1, \infty}^2\sum_{j=0}^1\left(\int_0^t\|\partial_y^j(v_2, p^{-1}-1, p-1)\|_{m-j}^2\; ds\right)^\frac12\left(\int_0^t\|  p-1\|_{m}^2 \;ds\right)^\frac12.
\end{align}
From the second equation in \eqref{3.1}, we have
 \begin{align}\label{change 30}
 \partial_y p-\varepsilon (2\mu+\lambda)\partial_y^2 v_2
 =-\rho \partial_t v_2+b_1\partial_x b_2-\rho {\bf{v}}\cdot \nabla v_2-b_1\partial_y b_1+\varepsilon\mu\partial_x^2 v_2+\varepsilon (\mu+\lambda)\partial_y\partial_x v_1.
 \end{align}
 Hence, the last part of \eqref{change 26} can be rewritten as
\begin{align}\label{change 27}
 &-\gamma^{-1}   \int_0^t\int_{\mathbb{R}_+^2} \mathcal{Z}^\alpha (p-1) \cdot   p^{-1}\mathcal{Z}^\alpha v_2\cdot\partial_y(p-1)\ d{\bf{x}}ds\notag\\
 =&-\varepsilon (2\mu+\lambda)\gamma^{-1}  \int_0^t\int_{\mathbb{R}_+^2}\mathcal{Z}^\alpha (p-1) \cdot   p^{-1}\mathcal{Z}^\alpha v_2\cdot\partial_y^2 v_2 \ d{\bf{x}}ds\notag\\
 &-\varepsilon (\mu+\lambda))\gamma^{-1}  \int_0^t\int_{\mathbb{R}_+^2}\mathcal{Z}^\alpha (p-1) \cdot   p^{-1}\mathcal{Z}^\alpha v_2\cdot\partial_y\partial_x v_1 \ d{\bf{x}}ds \notag\\
 &+\gamma^{-1}   \int_0^t\int_{\mathbb{R}_+^2} \mathcal{Z}^\alpha (p-1) \cdot   p^{-1}\mathcal{Z}^\alpha v_2\cdot\Big(
 \rho \partial_t v_2-b_1\partial_x b_2+\rho {\bf{v}}\cdot \nabla v_2-\varepsilon\mu\partial_x^2 v_2\Big)  \ d{\bf{x}}ds\notag\\
 &+\gamma^{-1}   \int_0^t\int_{\mathbb{R}_+^2} \mathcal{Z}^\alpha (p-1) \cdot   p^{-1}\mathcal{Z}^\alpha v_2\cdot b_1\partial_y b_1 \ d{\bf{x}}ds.
\end{align}
The first term can be estimated in the following way.
\begin{align*}
&-\varepsilon (2\mu+\lambda)\gamma^{-1}\int_0^t\int_{\mathbb{R}_+^2}\mathcal{Z}^\alpha (p-1)\cdot p^{-1}\mathcal{Z}^\alpha v_2\cdot\partial_y^2 v_2 \ d{\bf{x}}ds\\
\lesssim&\varepsilon \|\phi\partial_y^2v_2\|_{L_{t, {\bf{x}}}^\infty}\|p^{-1}\|_{L_{t, {\bf{x}}}^\infty}\left( \int_0^t\int_{\mathbb{R}_+^2}\|\mathcal{Z}^\alpha (p-1)\|_{L_x^2L_y^2}^2\;d{\bf{x}}ds\right)^{\frac12}   \left(\int_0^t\int_{\mathbb{R}_+^2}\|\phi^{-1}\mathcal{Z}^\alpha v_2\|_{L_x^2L_y^2}^2\;d{\bf{x}}ds\right)^{\frac12}\\
\le& \frac{\varepsilon\mu^2}{6}\int_0^t \|\partial_yv_2\|_{m}^2\;ds+C\varepsilon\|\partial_yv_2\|_{1, \infty}^2\|p^{-1}\|_{L_{t, {\bf{x}}}^\infty}^2\int_0^t \|p-1\|_m^2\;ds.
\end{align*}
Similarly, the second term has the bound
\begin{align*}
 \frac{\varepsilon\mu^2}{6}\int_0^t \|\partial_yv_2\|_{m}^2\;ds+C\varepsilon\|v_1\|_{2, \infty}^2\|p^{-1}\|_{L_{t, {\bf{x}}}^\infty}^2\int_0^t \|p^{-1}-1\|_m^2\;ds
\end{align*}
and the third term on the right hand side of \eqref{change 27} can be bounded by
\begin{align*}
\|p^{-1}\|_{L_{t, {\bf{x}}}^\infty}\Big(1+\|(\rho, {\bf{v}}, \partial_y v_2, {\bf{B}}-\overset{\rightarrow}{e_y})\|_{2, \infty}\Big)^3\int_0^t\|(  v_2, p-1 )\|_{m}^2 ds.
\end{align*}
To estimate the last term of \eqref{change 27},  by the first equation in \eqref{3.1}, we have
\begin{align}\label{change 31}
\partial_y b_1=\rho\partial_t v_1+\partial_x p+b_2\partial_x b_2-\tilde b_2\partial_y b_1+\rho {\bf{v}}\cdot\nabla v_1-\varepsilon \mu \partial_x^2 v_1-\varepsilon(\mu+\lambda)\partial_x(\nabla\cdot {\bf{v}})-\varepsilon\mu\partial_y^2 v_1,
\end{align}
 which yields that
\begin{align}\label{change 28}
 &\gamma^{-1}   \int_0^t\int_{\mathbb{R}_+^2} \mathcal{Z}^\alpha (p-1) \cdot   p^{-1}\mathcal{Z}^\alpha v_2\cdot b_1\partial_y b_1 \ d{\bf{x}}ds\notag\\
 = &-\varepsilon\mu\gamma^{-1}   \int_0^t\int_{\mathbb{R}_+^2} \mathcal{Z}^\alpha (p-1) \cdot   p^{-1}\mathcal{Z}^\alpha v_2\cdot b_1 \partial_y^2 v_1  \ d{\bf{x}}ds\notag\\
 &+ \gamma^{-1}   \int_0^t\int_{\mathbb{R}_+^2}\mathcal{Z}^\alpha (p-1) \cdot   p^{-1}\mathcal{Z}^\alpha v_2\cdot b_1 \Big(\rho\partial_t v_1+\partial_x p+b_2\partial_x b_2-\phi^{-1}\tilde b_2\phi\partial_y b_1\notag\\
 &+\rho  v_1 \partial_x v_1 +\rho \phi^{-1} v_2 \phi\partial_y v_1-\varepsilon \mu \partial_x^2 v_1-\varepsilon(\mu+\lambda)\partial_x(\nabla\cdot {\bf{v}}) \Big)\ d{\bf{x}}ds.
\end{align}
The first term can be handled as follows.
\begin{align*}
&-\varepsilon\mu\gamma^{-1}   \int_0^t\int_{\mathbb{R}_+^2} \mathcal{Z}^\alpha (p-1) \cdot   p^{-1}\mathcal{Z}^\alpha v_2\cdot b_1 \partial_y^2 v_1  \ d{\bf{x}}ds\notag\\
\lesssim&\varepsilon\|p^{-1}\|_{L_{t, {\bf{x}}}^\infty}\|b_1\|_{L_{t, {\bf{x}}}^\infty}\|\phi\partial_y^2 v_1 \|_{L_{t,{\bf{x}}}^\infty  }\left(\int_0^t\|\mathcal{Z}^\alpha (p-1)\|_{L_x^2L_y^2}^2\; ds\right)^\frac12\left(\int_0^t\|\phi^{-1}\mathcal{Z}^\alpha v_2\|_{L_x^2L_y^2}^2\; ds\right)^\frac12\notag\\
\lesssim&\frac{\varepsilon\mu^2}{6}\int_0^t \|\partial_yv_2\|_m^2\;ds
+\varepsilon\|\partial_y  v_1 \|_{1, \infty}^2\|p^{-1}\|_{L_{t, {\bf{x}}}^\infty}^2\|b_1\|_{L_{t, {\bf{x}}}^\infty}^2\int_0^t\| p-1\|_{m}^2\;ds.
\end{align*}
 The second term on the right hand side of \eqref{change 28} can be bounded by
 \begin{align*}
  \|p^{-1}\|_{L_{t, {\bf{x}}}^\infty} \|b_1\|_{L_{t, {\bf{x}}}^\infty}\left(1+\|(\rho, v_1, {\bf{B}}-\overset{\rightarrow}{e_y}, \partial_y v_2)\|_{1, \infty}\right)^3\int_0^t\|(v_2, p-1)\|_{m}^2 \;ds.
 \end{align*}
It remains to estimate the last two terms on the right hand side of \eqref{3.2}.
\begin{align}\label{change 22}
&\int_0^t\int_{\mathbb{R}_+^2} \mathcal{Z}^\alpha [(\nabla\times {\bf{B}})\times {\bf{B}}]\cdot \mathcal{Z}^\alpha {\bf{v}} \ d{\bf{x}}ds +\int_0^t\int_{\mathbb{R}_+^2} \mathcal{Z}^\alpha[\nabla\times({\bf{v}}\times {\bf{B}})]\cdot\mathcal{Z}^\alpha ({\bf{B}}-\overset{\rightarrow}{e_y})\  d{\bf{x}}ds\notag\\
=&-\sum\limits_{|\beta|+|\gamma|=|\alpha|}C_\alpha^\beta\int_0^t\int_{\mathbb{R}_+^2}\Big( \mathcal{Z}^\beta b_2\partial_x\mathcal{Z}^\gamma b_2\mathcal{Z}^\alpha v_1+ \mathcal{Z}^\beta v_1\partial_x\mathcal{Z}^\gamma b_1\mathcal{Z}^\alpha b_1  -\mathcal{Z}^\beta b_1\partial_x \mathcal{Z}^\gamma \widetilde  b_2\mathcal{Z}^\alpha v_2\notag\\
& - \mathcal{Z}^\beta b_2\partial_x\mathcal{Z}^\gamma v_1\mathcal{Z}^\alpha \widetilde b_2-\mathcal{Z}^\beta v_1\partial_x\mathcal{Z}^\gamma b_2 \mathcal{Z}^\alpha \widetilde b_2+\mathcal{Z}^\beta b_1\partial_x\mathcal{Z}^\gamma v_2\mathcal{Z}^\alpha \widetilde b_2+\mathcal{Z}^\beta v_2\partial_x\mathcal{Z}^\gamma b_1\mathcal{Z}^\alpha \widetilde b_2 \Big) \ d{\bf{x}}ds\notag\\ &+\sum\limits_{|\beta|+|\gamma|=|\alpha|}C_\alpha^\beta\int_0^t\int_{\mathbb{R}_+^2}\Big( \mathcal{Z}^\beta b_2 \mathcal{Z}^\gamma \partial_y b_1\mathcal{Z}^\alpha v_1 - \mathcal{Z}^\beta b_1 \mathcal{Z}^\gamma \partial_y b_1\mathcal{Z}^\alpha v_2 +  \mathcal{Z}^\beta b_2 \mathcal{Z}^\gamma \partial_y v_1\mathcal{Z}^\alpha b_1   \notag\\
&-\mathcal{Z}^\beta b_1 \mathcal{Z}^\gamma \partial_y v_2\mathcal{Z}^\alpha b_1-\mathcal{Z}^\beta v_2 \mathcal{Z}^\gamma \partial_y b_1 \mathcal{Z}^\alpha b_1 \Big) \;d{\bf{x}}ds \triangleq\mathbb{I}_1+\mathbb{I}_2.
\end{align}
 First we consider the case $\gamma=\alpha$. By integration by parts, we  write the terms on the right hand side of \eqref{change 22} as follows
 \begin{align}\label{change 23}
 &\int_0^t\int_{\mathbb{R}_+^2} \partial_x\widetilde b_2\mathcal{Z}^\alpha b_2\mathcal{Z}^\alpha v_1\  d{\bf{x}}ds+\int_0^t\int_{\mathbb{R}_+^2} \partial_x  b_1\mathcal{Z}^\alpha v_1\mathcal{Z}^\alpha b_1\ d{\bf{x}}ds-\int_0^t\int_{\mathbb{R}_+^2} \partial_x  b_1\mathcal{Z}^\alpha v_2\mathcal{Z}^\alpha \widetilde b_2\ d{\bf{x}}ds\notag\\
 &+\frac12\int_0^t\int_{\mathbb{R}_+^2}  \nabla\cdot {\bf{v}}|\mathcal{Z}^\alpha  ({\bf{B}}-\overset{\rightarrow}{e_y})|^2 \ d{\bf{x}}ds+\int_0^t\int_{\mathbb{R}_+^2} \partial_y  b_1\mathcal{Z}^\alpha b_1\mathcal{Z}^\alpha v_2\ d{\bf{x}}ds.
 \end{align}
 It is easy to bound the first four terms of \eqref{change 23}  by
 \begin{align*}
 \|(v_1, \partial_y v_2, b_1, \widetilde b_2)\|_{1, \infty} \int_0^t\|(v_1, v_2,  b_1, \widetilde b_2)\|_{m}^2\ ds.
 \end{align*}
 By \eqref{change 31} and the Sobolev embedding inequality,  the last term on the right hand side of \eqref{change 23} is solved by
 \begin{align}\label{change 25}
& \int_0^t\int_{\mathbb{R}_+^2} \partial_y  b_1\mathcal{Z}^\alpha b_1\mathcal{Z}^\alpha v_2\ d{\bf{x}}ds\notag\\
=&- \varepsilon\mu\int_0^t\int_{\mathbb{R}_+^2} \partial_y^2 v_1\mathcal{Z}^\alpha b_1\mathcal{Z}^\alpha v_2\ d{\bf{x}}ds  +\int_0^t\int_{\mathbb{R}_+^2}\Big(\rho\partial_t v_1+\partial_x p+b_2\partial_x b_2-\phi^{-1}\tilde b_2\phi\partial_y b_1\notag\\
&+\rho  v_1 \partial_xv_1+\rho  \phi^{-1}v_2 \phi\partial_yv_1 -\varepsilon \mu \partial_x^2 v_1-\varepsilon(\mu+\lambda)\partial_x(\nabla\cdot {\bf{v}})\Big)\mathcal{Z}^\alpha b_1\mathcal{Z}^\alpha v_2\ d{\bf{x}}ds\notag\\
\lesssim&\varepsilon\|\phi\partial_y^2 v_1\|_{L_{t, {\bf{x}}}^\infty}\left(\int_0^t\|\mathcal{Z}^\alpha b_1\|_{L_x^2L_y^2}^2\; ds\right)^\frac12\left(\int_0^t\|\phi^{-1}\mathcal{Z}^\alpha v_2\|_{L_x^2L_y^2}^2\; ds\right)^\frac12\notag\\
&+\left(1+\|(\rho, v_1, p, b_1, \tilde{b}_2, \partial_y v_2)\|_{2, \infty}\right)^3\int_0^t\|(v_2, b_1)\|_{m}^2\; ds \notag\\
\lesssim&\frac{\varepsilon\mu^2}{6}\int_0^t \|\partial_yv_2\|_m^2\;ds
+\varepsilon\|\partial_y v_1\|_{1, \infty}^2\int_0^t\|  b_1\|_{m}^2\; ds\notag\\
&+\left(1+\|({\bf{v}}, {\bf{B}}-\overset{\rightarrow}{e_y}, p,\partial_y v_2)\|_{2, \infty}\right)^3\int_0^t\|(v_2, b_1)\|_{m}^2 \;ds.
 \end{align}
 Next, we  consider the case $|\gamma|<|\alpha|$.  It is easy to know  $\mathbb{I}_1$  has the bound
 \begin{align*}
 \|({\bf{v}}, {\bf{B}}-\overset{\rightarrow}{e_y})\|_{1, \infty}\int_0^t\|({\bf{v}}, {\bf{B}}-\overset{\rightarrow}{e_y})\|_{m}^2 \;ds.
 \end{align*}
 Special attention is paid to  $\mathbb{I}_2$.  When  $1\leq |\beta|<|\gamma|$,    it  is  bounded by
 \begin{align*}
&\|( {\bf{v}},  {\bf{B}}-\overset{\rightarrow}{e_y})\|_{ [m/2], \infty}\left(\int_0^t\|( \partial_y {\bf{v}},  \partial_y{\bf{B}})\|_{m-1}^2\; ds\right)^\frac12 \left(\int_0^t\|( {\bf{v}},   {\bf{B}}-\overset{\rightarrow}{e_y})\|_{m}^2 \;ds\right)^\frac12.
 \end{align*}
When   $|\gamma|\leq |\beta|<|\alpha|$,  the Sobolev embedding inequality yields that
 \begin{align*}
 \mathbb{I}_2\lesssim&\sum_{\beta+\gamma=\alpha\atop |\gamma|\leq |\beta|<|\alpha| }\sup_{0\leq s\leq t}\|(\mathcal{Z}^\gamma \partial_y {\bf{v}}, \mathcal{Z}^\gamma\partial_y {\bf{B}} )(s)\|_{L_x^\infty L_y^2}\left(\int_0^t\|(\mathcal{Z}^\beta {\bf{v}}, \mathcal{Z}^\beta({\bf{B}}-\overset{\rightarrow}{e_y})\|_{L_{x}^2 L_y^\infty}^2\; ds\right)^\frac12\notag\\
 & \cdot\left(\int_0^t\|(\mathcal{Z}^\alpha  {\bf{v}}, \mathcal{Z}^\alpha( {\bf{B}}-\overset{\rightarrow}{e_y}))\|_{L_x^2L_y^2}^2\; ds\right)^\frac12\notag\\
 \lesssim&\left[\|(\partial_y {\bf{v}}(0), \partial_y {\bf{B}} (0))\|_{[m/2]+2}+\left(\int_0^t\|(\partial_y {\bf{v}}, \partial_y {\bf{B}})\|_{[m/2]+2}^2\; ds\right)^\frac12\right] \left(\int_0^t\| ({\bf{v}},   {\bf{B}}-\overset{\rightarrow}{e_y})\|_{m-1}^2\; ds\right)^\frac14\notag\\
 &\cdot\left(\int_0^t\| (\partial_y {\bf{v}},   \partial_y{\bf{B}} )\|_{m-1}^2 \;ds\right)^\frac14 \left(\int_0^t\|({\bf{v}}, {\bf{B}}-\overset{\rightarrow}{e_y})\|_{m}^2 \;ds\right)^\frac12\\
\lesssim&\left(1+\|(\partial_y {\bf{v}}(0), \partial_y {\bf{B}} (0))\|_{[m/2]+2}^2+\int_0^t\|(\partial_y {\bf{v}}, \partial_y {\bf{B}})\|_{[m/2]+2}^2 \; ds\right) \\
 &\cdot\left(\int_0^t\| (\partial_y {\bf{v}}, \partial_y{\bf{B}} )\|_{m-1}^2 \;ds+ \int_0^t\|({\bf{v}}, {\bf{B}}-\overset{\rightarrow}{e_y})\|_{m}^2 \;ds\right).
 \end{align*}
 When $\beta=\alpha$, we write $\mathbb{I}_2$ as follows
\begin{align}\label{change 24}
&-\int_0^t\int_{\mathbb{R}_+^2}\partial_y v_2\mathcal{Z}^\alpha b_1  \mathcal{Z}^\alpha b_1 \ d{\bf{x}}ds
+\int_0^t\int_{\mathbb{R}_+^2} \partial_y b_1\left(\mathcal{Z}^\alpha b_2   \mathcal{Z}^\alpha v_1-2\mathcal{Z}^\alpha b_1    \mathcal{Z}^\alpha v_2  \right) \ d{\bf{x}}ds\notag\\
&+\int_0^t\int_{\mathbb{R}_+^2} \partial_y v_1 \mathcal{Z}^\alpha b_2   \mathcal{Z}^\alpha b_1 \ d{\bf{x}}ds.
\end{align}
The first term of \eqref{change 24} can be handled by
\begin{align*}
-\int_0^t\int_{\mathbb{R}_+^2}\partial_y v_2\mathcal{Z}^\alpha b_1  \mathcal{Z}^\alpha b_1 \ d{\bf{x}}ds\lesssim\|\partial_y v_2\|_{L_{t, {\bf{{x}}}}^\infty}\int_0^t\|b_1\|_{m}^2 \;ds.
\end{align*}
The second term has the similar estimates as \eqref{change 25}. For the last term on the right hand side of \eqref{change 24}, by \eqref{4.1}, we have
 \begin{align*}
&\int_0^t\int_{\mathbb{R}_+^2}  \partial_y v_1\mathcal{Z}^\alpha b_2    \mathcal{Z}^\alpha b_1 \ d{\bf{x}}ds\\
=&-\varepsilon\kappa\int_0^t\int_{\mathbb{R}_+^2} \partial_y^2 b_1 \mathcal{Z}^\alpha b_2  \mathcal{Z}^\alpha b_1 \ d{\bf{x}}ds-\varepsilon\kappa\int_0^t\int_{\mathbb{R}_+^2}\partial_x^2 b_1\mathcal{Z}^\alpha b_2    \mathcal{Z}^\alpha b_1 \ d{\bf{x}}ds\\
&+\int_0^t\int_{\mathbb{R}_+^2}\left(\partial_t b_1+v_1 \partial_xb_1-  \tilde b_2\partial_yv_1 +v_2\partial_y b_1+b_1\partial_y v_2\right)\mathcal{Z}^\alpha b_2\mathcal{Z}^\alpha b_1 \ d{\bf{x}}ds\\
\lesssim& \varepsilon\sup_{0\leq s\leq t}\|\phi\partial_y^2 b_1(s)\|_{L_{ x}^\infty L_y^2}\left(\int_0^t\|\phi^{-1}\mathcal{Z}^\alpha \widetilde b_2\|_{L_x^2L_y^2}^2\; ds\right)^\frac12\left(\int_0^t\|\mathcal{Z}^\alpha   b_1\|_{L_x^2L_y^\infty}^2\; ds\right)^\frac12\\
&+\left(1+\|(v_1, b_1,\phi^{-1}v_2, \phi\partial_y b_2, \partial_y v_2)\|_{2, \infty}\right)^2\int_0^t\| {\bf{B}}-\overset{\rightarrow}{e_y}\|_{m}^2 \;ds\\
\lesssim&\frac{\varepsilon\kappa^2}{6}\int_0^t \|\nabla {\bf{B}}\|_m^2\;ds +
\left(\|\partial_y b_1(0)\|_{3}^4+\int_0^t\|\partial_y b_1\|_{3}^4  \ ds\right)\int_0^t\|  b_1\|_{m}^2\;ds\\
& +\left(1+\|(v_1, b_1,  \tilde{b}_2, \partial_y v_2)\|_{2, \infty}\right)^2\int_0^t\| {\bf{B}}-\overset{\rightarrow}{e_y}\|_{m}^2 \;ds.
\end{align*}
Collecting all the above estimates and by H\"older's inequality, we prove Lemma \ref{lem1}.

\end{proof}

\section{Normal Derivative Estimates}
To close the energy estimates in Section 2,  it suffices to derive the conormal estimates for the first order normal derivatives of  $({\bf{v}}, {\bf{B}}, p)$ and the second order normal derivatives of $v_2$ due to Lemma \ref{L2}, which is carried out in the subsequent parts.

\subsection{Conormal Estimate of $\partial_y v_2 $}
We first consider the conormal estimate of $\partial_y v_2$. By the equation of density, one has
\begin{align}\label{4.7}
\partial_y v_2=-\partial_x v_1-  \gamma^{-1}p^{-1}\partial_t p- \gamma^{-1}p^{-1}{\bf{v}}\cdot \nabla p.
\end{align}
For any multi-index $\alpha$ satisfying $|\alpha|\leq m-1$, by applying  $\mathcal{Z}^\alpha $ to the above equality and taking the $L^2$ inner product on both sides of the resulting equality, it follows that
\begin{equation}\label{change 8}
\begin{split}
\int_0^t\int_{\mathbb{R}_+^2}|\mathcal{Z}^\alpha\partial_y v_2|^2 \;d{\bf{x}}ds
\lesssim& \int_0^t\int_{\mathbb{R}_+^2}|\mathcal{Z}^\alpha \partial_x v_1|^2\;d{\bf{x}}ds+ \int_0^t\int_{\mathbb{R}_+^2} |\mathcal{Z}^\alpha ( p^{-1}\partial_t p) |^2\;d{\bf{x}}ds\\
&+ \int_0^t\int_{\mathbb{R}_+^2} | \mathcal{Z}^\alpha (p^{-1} v_1\cdot \partial_x p) |^2\; d{\bf{x}}ds+ \int_0^t\int_{\mathbb{R}_+^2} | \mathcal{Z}^\alpha (p^{-1} v_2\cdot \partial_y p) |^2\; d{\bf{x}}ds.
\end{split}
\end{equation}
The first three terms on the right hand side of \eqref{change 8} share the same estimates as \cite{CLX21}. On the other hand, by the same trick as \eqref{aa}, the last term has the bound.

\begin{align*}
& \int_0^t\int_{\mathbb{R}_+^2} | \mathcal{Z}^\alpha (p^{-1} v_2\cdot \partial_y p) |^2 d{\bf{x}}ds\\
\lesssim&\|p^{-1}\|_{L_{t, {\bf{x}}}^\infty}^2\|\partial_yv_2\|_{L_{t, {\bf{x}}}^\infty}^2\int_0^t \|p-1\|_m^2\;ds+\sup_{0\leq s\leq t}\|\partial_yp(s)\|_{L_{x} ^\infty L_y^2}^2 \|p^{-1}\|_{L_{t, {\bf{x}}}^\infty}^2\int_0^t \|\mathcal{Z}^\alpha v_2\|_{L_x^2L_y^\infty}^2\;ds\\
 &+\|p\|_{1, \infty}^2\|\partial_yv_2\|_{L_{t, {\bf{x}}}^\infty}^2\int_0^t \|p^{-1}-1\|_{m-1}^2\;ds.
\end{align*}
Notice that by  the Sobolev embedding inequality, the second term can be estimated as follows.
\begin{align*}
&\sup_{0\leq s\leq t}\|\partial_yp(s)\|_{L_{x} ^\infty L_y^2}^2 \|p^{-1}\|_{L_{t, {\bf{x}}}^\infty}^2\int_0^t \|\mathcal{Z}^\alpha v_2\|_{L_x^2L_y^\infty}^2\;ds\\
\le &C \|p^{-1}\|_{L_{t, {\bf{x}}}^\infty}^4 \left(\|\partial_yp(0)\|_{2}^4+\int_0^t \|\partial_yp\|_{2}^4\;ds\right)\int_0^t \|v_2\|_{m-1}^2\;ds+\frac12\int_0^t \|\partial_yv_2\|_{m-1}^2\;ds.
\end{align*}
Thus we obtain the estimate of $\partial_yv_2$.
\begin{align*}
\int_0^t\int_{\mathbb{R}_+^2}|\mathcal{Z}^\alpha\partial_y v_2|^2 \;d{\bf{x}}ds
\lesssim& \left(1+\|(v_1, p, p^{-1}, \partial_yv_2 )\|_{1, \infty}^2\right)^2\int_0^t\|(v_1, p-1, p^{-1}-1)\|_{m}^2 \;ds\\
&+ \|p^{-1}\|_{L_{t, {\bf{x}}}^\infty}^4\left(\|\partial_yp(0)\|_{2}^4+\int_0^t \|\partial_yp\|_{2}^4\;ds\right) \int_0^t \|v_2\|_{m-1}^2\;ds.
\end{align*}

\subsection{ Conormal Estimate of $\partial_y v_1 $}
Then we consider the conormal estimate of $\partial_y v_1$. 	For any multi-index $\alpha$ satisfying $|\alpha|\leq m-1$, by applying $\mu^\frac12\mathcal{Z}^\alpha  $ to  \eqref{4.1} and taking  the $L^2$ inner product of the  resulting equality, we have
\begin{align}\label{4.2}
&\int_0^t\int_{\mathbb{R}_+^2}\left(\mu| \mathcal{Z}^\alpha \partial_y v_1|^2\  +\varepsilon^2\kappa^2\mu |\mathcal{Z}^\alpha \partial_y^2 b_1|^2\right)\; d{\bf{x}}ds+2\varepsilon\kappa\mu\int_0^t\int_{\mathbb{R}_+^2} \mathcal{Z}^\alpha \partial_y v_1\cdot\mathcal{Z}^\alpha \partial_y^2 b_1\; d{\bf{x}}ds \nonumber \\
&\lesssim\varepsilon^2\kappa^2\mu\int_0^t\int_{\mathbb{R}_+^2}|\mathcal{Z}^\alpha  \partial_x^2 b_1|^2 \;d{\bf{x}}ds+ \mu\int_0^t\int_{\mathbb{R}_+^2}|	\partial_t\mathcal{Z}^\alpha  b_1|^2 \;d{\bf{x}}ds +\mu\int_0^t\int_{\mathbb{R}_+^2}|\mathcal{Z}^\alpha  (v_1\partial_x b_1)|^2 \;d{\bf{x}}ds \nonumber\\
&+\mu\int_0^t\int_{\mathbb{R}_+^2}|\mathcal{Z}^\alpha  (\tilde{b}_2\partial_y v_1)|^2\; d{\bf{x}}ds+\mu\int_0^t\int_{\mathbb{R}_+^2}|\mathcal{Z}^\alpha  (v_2\partial_y b_1 )|^2\; d{\bf{x}}ds +\mu\int_0^t\int_{\mathbb{R}_+^2}|\mathcal{Z}^\alpha  (b_1\partial_y v_2)|^2\; d{\bf{x}}ds\nonumber\\
&\triangleq \varepsilon^2\kappa^2\mu\int_0^t\int_{\mathbb{R}_+^2}|\mathcal{Z}^\alpha  \partial_x^2 b_1|^2\; d{\bf{x}}ds+\mathbb{J}.
\end{align}
By integration by parts,   we handle the second term on the left hand side of \eqref{4.2}  as follows.
\begin{align}\label{4.3}
&2\varepsilon\kappa\mu\int_0^t\int_{\mathbb{R}_+^2} \mathcal{Z}^\alpha \partial_y v_1\cdot\mathcal{Z}^\alpha \partial_y^2 b_1\; d{\bf{x}}ds\notag\\
=&2\varepsilon\kappa\mu\int_0^t\int_{\mathbb{R}_+^2} \mathcal{Z}^\alpha \partial_y v_1\cdot[\mathcal{Z}^\alpha, \partial_y] \partial_y b_1\; d{\bf{x}}ds -2\varepsilon\kappa\mu\int_0^t\int_{\mathbb{R}_+^2} [\partial_y, \mathcal{Z}^\alpha] \partial_y v_1\cdot\mathcal{Z}^\alpha \partial_y b_1 \;d{\bf{x}}ds\notag\\
&-2\varepsilon\kappa\mu\int_0^t\int_{\mathbb{R}_+^2}  \mathcal{Z}^\alpha \partial_y^2 v_1\cdot\mathcal{Z}^\alpha \partial_y b_1\; d{\bf{x}}ds,
\end{align}
where the boundary condition $\partial_yb_1|_{y=0}=0$ is used.

For the first term on the right hand side of \eqref{4.3},  by \eqref{2.2},  one has
\begin{align*}
&2\varepsilon\kappa\mu\int_0^t\int_{\mathbb{R}_+^2} \mathcal{Z}^\alpha \partial_y v_1\cdot[\mathcal{Z}^\alpha, \partial_y] \partial_y b_1\; d{\bf{x}}ds  \\
=&2\varepsilon\kappa\mu\sum_{k=0}^{m-2}\int_0^t\int_{\mathbb{R}_+^2} \phi_{k,m-1}(y) \mathcal{Z}^\alpha \partial_y v_1\cdot \mathcal{Z}_y^k \partial_y^2 b_1 \;d{\bf{x}}ds\\
\le& \frac\mu 2\int_0^t \|\partial_yv_1\|_{m-1}^2\;ds+C\varepsilon^2\kappa^2\mu^2\int_0^t \|\partial_y^2b_1\|_{m-2}^2\;ds.
\end{align*}
With the similar line, the second term on the right hand side of \eqref{4.3} has the bound
\begin{align*}
&-2\varepsilon\kappa\mu\int_0^t\int_{\mathbb{R}_+^2} [\partial_y, \mathcal{Z}^\alpha] \partial_y v_1\cdot\mathcal{Z}^\alpha \partial_y b_1 \;d{\bf{x}}ds \\
=&-2\varepsilon\kappa\mu\sum_{k=0}^{m-2}\int_0^t\int_{\mathbb{R}_+^2}  \phi^{k, m-1}(y) \mathcal{Z}_y^k\partial_y^2 v_1\cdot\mathcal{Z}^\alpha \partial_y b_1 \;d{\bf{x}}ds \\
\le& \frac\kappa 2\int_0^t \|\partial_yb_1\|_{m-1}^2\;ds+C\varepsilon^2\kappa^2\mu^2\int_0^t \|\partial_y^2v_1\|_{m-2}^2\;ds.
\end{align*}
Next, we turn to consider  the terms on the right hand side of \eqref{4.2}. For the first term on the right hand side of \eqref{4.2}, we have
\begin{align*}
\varepsilon^2\kappa^2\mu\int_0^t\int_{\mathbb{R}_+^2}|\mathcal{Z}^\alpha  \partial_x^2 b_1|^2\; d{\bf{x}}ds\lesssim	\varepsilon^2\kappa^2\mu\int_0^t\|\partial_xb_1\|_{m}^2 \;ds.
\end{align*}
With the same calculations in \cite{CLX21}, the rest terms on the right hand side of \eqref{4.2} can be  estimated as follows.
\begin{align*}
\mathbb{J}\lesssim \left(1+  \|(v_1, b_1, \partial_y v_2)\|_{1,\infty}^2\right)\int_0^t\|( {\bf{v}}, b_1 )\|_{m}^2 \;ds +\|b_1\|_{1, \infty}^2\int_0^t\|\partial_y v_2\|_{m-1}^2 \;ds.
\end{align*}
Combining above estimates, we obtain by induction that
\begin{align*}
&\int_0^t\int_{\mathbb{R}_+^2}\left(\mu| \mathcal{Z}^\alpha \partial_y v_1|^2\  +\varepsilon^2\kappa^2\mu |\mathcal{Z}^\alpha \partial_y^2 b_1|^2\right) \;d{\bf{x}}ds
-2\varepsilon\mu\kappa\int_0^t\int_{\mathbb{R}_+^2} \mathcal{Z}^\alpha \partial_y^2 v_1\cdot\mathcal{Z}^\alpha \partial_y b_1 \;d{\bf{x}}ds\\
\le &\frac\kappa 2\int_0^t \|\partial_yb_1\|_{m-1}^2\;ds+C\left(1+  \|(v_1, b_1, \partial_y v_2)\|_{1,\infty}^2\right)\int_0^t\|( {\bf{v}}, b_1 )\|_{m}^2 \;ds\notag\\
&  +C\varepsilon^2\kappa^2\mu^2\int_0^t \|\partial_y^2v_1\|_{m-2}^2\;ds+C\|b_1\|_{1, \infty}^2\int_0^t\|\partial_y v_2\|_{m-1}^2\; ds.
\end{align*}

 \subsection{Conormal Estimate of $\partial_y b_1 $ }
To derive the conormal estimate of $\partial_y b_1$, the transverse background magnetic field is also essentially used.
For any  multi-index $\alpha$ satisfying $|\alpha|\leq m-1$, applying  $\kappa^\frac12\mathcal{Z}^\alpha $ to \eqref{change 31} and taking the $L^2$ inner product of the resulting equality give that
\begin{align}\label{4.6}
&\int_0^t\int_{\mathbb{R}_+^2}\left( \kappa|\mathcal{Z}^\alpha\partial_y b_1|^2+\varepsilon^2\mu^2\kappa|\mathcal{Z}^\alpha\partial_y^{2}  v_1|^2 \right)d{\bf{x}}ds +2\varepsilon\mu\kappa\int_0^t\int_{\mathbb{R}_+^2}\mathcal{Z}^\alpha\partial_y^{2}  v_1 \cdot \mathcal{Z}^\alpha\partial_y b_1\; d{\bf{x}}ds \notag\\
\lesssim&\kappa\int_0^t\int_{\mathbb{R}_+^2} |\mathcal{Z}^\alpha (\rho\partial_t v_1)|^2\; d{\bf{x}}ds+\kappa\int_0^t\int_{\mathbb{R}_+^2} |\mathcal{Z}^\alpha \partial_x p|^2\; d{\bf{x}}ds+\kappa\int_0^t\int_{\mathbb{R}_+^2} |\mathcal{Z}^\alpha (b_2\partial_x b_2)|^2\; d{\bf{x}}ds \notag\\
&+\varepsilon^2 \mu ^2\kappa\int_0^t\int_{\mathbb{R}_+^2}|\mathcal{Z}^\alpha \partial_x^2 v_1|^2\;d{\bf{x}}ds+\varepsilon^2(\mu+\lambda)^2\kappa\int_0^t\int_{\mathbb{R}_+^2} |\mathcal{Z}^\alpha \partial_x (\nabla\cdot {\bf{v}})|^2\; d{\bf{x}}ds\notag\\
&+\kappa\int_0^t\int_{\mathbb{R}_+^2} |\mathcal{Z}^\alpha (\phi^{-1}\tilde b_2\phi\partial_y b_1)|^2\; d{\bf{x}}ds+\kappa\int_0^t\int_{\mathbb{R}_+^2} |\mathcal{Z}^\alpha (\rho  v_1 \partial_x v_1)|^2\; d{\bf{x}}ds\notag\\
&+\kappa\int_0^t\int_{\mathbb{R}_+^2} |\mathcal{Z}^\alpha (\rho  v_2 \partial_y v_1)|^2\; d{\bf{x}}ds\triangleq\mathbb{K}+\kappa\int_0^t\int_{\mathbb{R}_+^2} |\mathcal{Z}^\alpha (\rho  v_2 \partial_y v_1)|^2 \;d{\bf{x}}ds.
\end{align}	
 By the same arguments as in \cite{CLX21},  we solve  $\mathbb{K}$ by
\begin{align*}
\mathbb{K}\lesssim&\left(1+\|(\rho, v_1, b_1, \tilde{b}_2)\|_{1, \infty}^2\right)^2\int_0^t\|(v_1, b_1, \tilde{b}_2, p-1)\|_{m}^2\; ds+\varepsilon^2\mu^2\kappa\int_0^t\|\partial_x v_1\|_{m}^2\; ds\notag\\
&+\varepsilon^2(\mu+\lambda)^2\int_0^t\|\nabla\cdot {\bf{v}}\|_{m}^2\; ds.
\end{align*}
For the last term on the right hand side of \eqref{4.6}, by Lemma \ref{lem 2.2}, we have
\begin{align}\label{aa}
 &\kappa\int_0^t\int_{\mathbb{R}_+^2} |\mathcal{Z}^\alpha (\rho v_2\partial_y v_1)|^2\ d{\bf{x}}ds \notag\\
 \lesssim & \int_0^t\int_{\mathbb{R}_+^2} |\mathcal{Z}^\alpha (\rho \phi^{-1}v_2\phi\partial_y v_1)|^2 \ d{\bf{x}}ds\notag\\
  \lesssim & \|\rho\phi^{-1}v_2\|_{L_{t, {\bf{x}}}^\infty}^2\int_0^t \|v_1\|_m^2\;ds+ \|\mathcal{Z}v_1\|_{L_{t, {\bf{x}}}^\infty}^2\int_0^t \|\rho\phi^{-1}v_2\|_{m-1}^2\;ds\notag\\
 \lesssim &\|\rho\|_{L_{t, {\bf{x}}}^\infty}^2\|\partial_yv_2\|_{L_{t, {\bf{x}}}^\infty}^2\int_0^t \|v_1\|_m^2\;ds+\|v_1\|_{1, \infty}^2\|\rho\|_{L_{t, {\bf{x}}}^\infty}^2\int_0^t \|\partial_yv_2\|_{m-1}^2\;ds\notag\\
 &+\|v_1\|_{1, \infty}^2\|\partial_yv_2\|_{L_{t, {\bf{x}}}^\infty}^2\int_0^t \|p-1\|_{m-1}^2\;ds.
\end{align}
 Thus we conclude that
 \begin{align*}
 &\int_0^t\int_{\mathbb{R}_+^2}\left( \kappa|\mathcal{Z}^\alpha\partial_y b_1|^2+\varepsilon^2\mu^2\kappa|\mathcal{Z}^\alpha\partial_y^{2}  v_1|^2 \right)\ d{\bf{x}}ds +2\varepsilon\mu\kappa\int_0^t\int_{\mathbb{R}_+^2}\mathcal{Z}^\alpha\partial_y^{2}  v_1 \cdot \mathcal{Z}^\alpha\partial_y b_1 \ d{\bf{x}}ds\\
 \lesssim&\left(1+\|(\rho, v_1, b_1, \tilde{b}_2, \partial_yv_2)\|_{1, \infty}^2\right)^2\int_0^t\|(\rho, v_1,  {\bf{B}}-\overset{\rightarrow}{e_y}, p-1)\|_{m}^2 \ ds+\|v_1\|_{1, \infty}^2\int_0^t \|\partial_yv_2\|_{m-1}^2\;ds\\
 &+\varepsilon^2\mu^2\kappa\int_0^t\|\partial_x v_1\|_{m}^2 \ ds+\varepsilon^2(\mu+\lambda)^2\int_0^t\|\nabla\cdot {\bf{v}}\|_{m}^2\ ds.
 \end{align*}

\subsection{Conormal Estimate of $\partial_y p $ }
This subsection is devoted to the conormal estimate of $\partial_y p$.
For any  multi-index $\alpha$ satisfying $|\alpha|\leq m-1$, by applying $\mathcal{Z}^\alpha $ to  \eqref{change 30} and taking the  $L^2$ inner product on both sides of the resulting equation, we have
\begin{equation}\label{change 12}
\begin{split}
&\int_0^t\int_{\mathbb{R}_+^2}\left( |\mathcal{Z}^\alpha \partial_y p |^2+ \varepsilon ^2(2\mu+\lambda)^2 |\mathcal{Z}^\alpha \partial_y^2 v_2|^2\right)  \ d{\bf{x}}ds -2\varepsilon (2\mu+\lambda)\int_0^t\int_{\mathbb{R}_+^2}\mathcal{Z}^\alpha  \partial_y^2  v_2\cdot \mathcal{Z}^\alpha \partial_y p \ d{\bf{x}}ds \\
\lesssim &\int_0^t\int_{\mathbb{R}_+^2} |\mathcal{Z}^\alpha(\rho \partial_t v_2)|^2\ d{\bf{x}}ds+\int_0^t\int_{\mathbb{R}_+^2} |\mathcal{Z}^\alpha (b_1\partial_x b_2)|^2\ d{\bf{x}}ds\\
&+\varepsilon^2\mu^2\int_0^t\int_{\mathbb{R}_+^2} |\mathcal{Z}^\alpha \partial_x^2 v_2|^2 \ d{\bf{x}}ds  +\varepsilon^2 (\mu+\lambda)^2\int_0^t\int_{\mathbb{R}_+^2}|\mathcal{Z}^\alpha\partial_y\partial_x v_1|^2 \ d{\bf{x}}ds \\
&+\int_0^t\int_{\mathbb{R}_+^2} |\mathcal{Z}^\alpha (\rho {\bf{v}}\cdot \nabla v_2)|^2\ d{\bf{x}}ds +\int_0^t\int_{\mathbb{R}_+^2} |\mathcal{Z}^\alpha  (b_1\partial_y b_1)|^2\ d{\bf{x}}ds.
\end{split}
\end{equation}
We first consider the second term on the left hand side of \eqref{change 12}.
 By multiplying   the first equation in \eqref{3.1} with $\gamma \rho^{\gamma-1}$ and  applying $\partial_y$ to the  resulting equality, we have
\begin{align*}
-\partial_y^2 v_2=&\gamma^{-1}p^{-1}\partial_t\partial_y p+p^{-1}\partial_y p(\partial_x v_1+\partial_y v_2)+\partial_x\partial_y v_1+\gamma^{-1}p^{-1}\partial_y v_1\partial_x p\notag\\
&+\gamma^{-1}p^{-1}v_1\partial_x\partial_y p+\gamma^{-1}p^{-1}\partial_y v_2 \partial_y p+\gamma^{-1}p^{-1}  v_2 \partial_y^2 p.
\end{align*}
Substitute it into the second term on the left hand side of \eqref{change 12}  to get
 \begin{align}\label{change 20}
 &-2\varepsilon (2\mu+\lambda)\int_0^t\int_{\mathbb{R}_+^2}\mathcal{Z}^\alpha  \partial_y^2  v_2\cdot \mathcal{Z}^\alpha \partial_y p \ d{\bf{x}}ds\notag\\
 =&2\varepsilon (2\mu+\lambda)\gamma^{-1}\int_0^t\int_{\mathbb{R}_+^2}\mathcal{Z}^\alpha  (p^{-1}\partial_t\partial_y p)\cdot \mathcal{Z}^\alpha \partial_y p \ d{\bf{x}}ds\notag\\
 &+2\varepsilon (2\mu+\lambda) \int_0^t\int_{\mathbb{R}_+^2}\mathcal{Z}^\alpha [p^{-1} \partial_y p(\partial_x v_1+\partial_y v_2)]\cdot \mathcal{Z}^\alpha \partial_y p \ d{\bf{x}}ds\notag\\
  &+2\varepsilon (2\mu+\lambda) \int_0^t\int_{\mathbb{R}_+^2}\partial_x\mathcal{Z}^\alpha\partial_y v_1\cdot \mathcal{Z}^\alpha \partial_y p \ d{\bf{x}}ds\notag\\
  &+2\varepsilon (2\mu+\lambda)\gamma^{-1} \int_0^t\int_{\mathbb{R}_+^2}\mathcal{Z}^\alpha(p^{-1}\partial_y v_1\partial_x p)\cdot \mathcal{Z}^\alpha \partial_y p \ d{\bf{x}}ds\notag\\
   &+2\varepsilon (2\mu+\lambda)\gamma^{-1} \int_0^t\int_{\mathbb{R}_+^2}\mathcal{Z}^\alpha(p^{-1} v_1\partial_x \partial_yp)\cdot \mathcal{Z}^\alpha \partial_y p \ d{\bf{x}}ds\notag\\
   &+2\varepsilon (2\mu+\lambda)\gamma^{-1} \int_0^t\int_{\mathbb{R}_+^2}\mathcal{Z}^\alpha(p^{-1}\partial_y v_2  \partial_yp)\cdot \mathcal{Z}^\alpha \partial_y p \ d{\bf{x}}ds\notag\\
   &+2\varepsilon (2\mu+\lambda)\gamma^{-1} \int_0^t\int_{\mathbb{R}_+^2}\mathcal{Z}^\alpha(p^{-1}v_2  \partial_y^2p)\cdot \mathcal{Z}^\alpha \partial_y p \ d{\bf{x}}ds.
 \end{align}
  We separate the first term on the right hand side of \eqref{change 20} into two parts.
 \begin{align}\label{change 21}
 &2\varepsilon (2\mu+\lambda)\gamma^{-1}\int_0^t\int_{\mathbb{R}_+^2}\mathcal{Z}^\alpha  (p^{-1}\partial_t\partial_y p)\cdot \mathcal{Z}^\alpha \partial_y p \ d{\bf{x}}ds\notag\\
 =&2\varepsilon (2\mu+\lambda)\gamma^{-1}\sum_{\beta+\gamma=\alpha\atop |\beta|\geq 1}\int_0^t\int_{\mathbb{R}_+^2}\mathcal{Z}^\beta p^{-1}\partial_t\mathcal{Z}^\gamma\partial_y p\cdot \mathcal{Z}^\alpha \partial_y p \ d{\bf{x}}ds\notag\\
 &+2\varepsilon (2\mu+\lambda)\gamma^{-1} \int_0^t\int_{\mathbb{R}_+^2} p^{-1}\partial_t\mathcal{Z}^\alpha\partial_y p\cdot \mathcal{Z}^\alpha \partial_y p \ d{\bf{x}}ds.
 \end{align}
 By the Sobolev embedding inequality,  the first term can be dealt by
 \begin{align*}
 &2\varepsilon (2\mu+\lambda)\gamma^{-1}\sum_{\beta+\gamma=\alpha\atop |\beta|\geq 1}\int_0^t\int_{\mathbb{R}_+^2}\mathcal{Z}^\beta p^{-1}\partial_t\mathcal{Z}^\gamma\partial_y p\cdot \mathcal{Z}^\alpha \partial_y p \ d{\bf{x}}ds\notag\\
 \lesssim&\varepsilon  \sum_{\beta+\gamma=\alpha\atop |\gamma|\leq|\beta|  }\sup_{0\leq s\leq t}\|\partial_t\mathcal{Z}^\gamma\partial_y p(s)\|_{L_{x}^\infty L_y^2}\left(\int_0^t\|\mathcal{Z}^\beta p^{-1}\|_{L_x^2L_y^\infty}^2\; ds\right)^\frac12\left(\int_0^t\|\mathcal{Z}^\alpha \partial_y p\|_{L_x^2L_y^2}^2\; ds\right)^\frac12\notag\\
 &+\varepsilon  \sum_{\beta+\gamma=\alpha\atop |\beta|< |\gamma|}\|\mathcal{Z}^\beta p^{-1}\|_{L_{t, {\bf{x}}}^\infty  }\left(\int_0^t\|\partial_t\mathcal{Z}^\gamma\partial_y p\|_{L_x^2L_y^2}^2\; ds\right)^\frac12\left(\int_0^t\|\mathcal{Z}^\alpha \partial_y p\|_{L_x^2L_y^2}^2\; ds\right)^\frac12\notag\\
  \lesssim&\varepsilon \left[\|\partial_y p(0)\|_{[(m-1)/2]+3}+\left(\int_0^t\| \partial_y p\|_{[(m-1)/2]+3}^2\; ds\right)^\frac12 \right]\left(\int_0^t\|p^{-1}-1\|_{m-1}^2 \;ds\right)^\frac14\notag\\
 &\cdot\left(\int_0^t\|\partial_y p^{-1}\|_{m-1}^2\; ds\right)^\frac14\left(\int_0^t\|\partial_y p \|_{m-1}^2\; ds\right)^\frac12+\varepsilon\|  p^{-1}\|_{[(m-1)/2], \infty}\int_0^t\|\partial_y p\|_{m-2}^2 \; ds\\
  \lesssim&\varepsilon^2 \left(\|\partial_y p(0)\|_{[(m-1)/2]+3}^2+\int_0^t\| \partial_y p\|_{[(m-1)/2]+3}^2\;  ds \right)\cdot\sum_{j=0}^1\int_0^t\|\partial_y ^j(p^{-1}-1)\|_{m-1}^2 \;ds\notag\\
 &+\varepsilon\|  p^{-1}\|_{[(m-1)/2], \infty}\int_0^t\|\partial_y p\|_{m-2}^2\; ds+\frac18\int_0^t \|\partial_yp\|_{m-1}^2\;ds.
 \end{align*}
 The second term on the right hand side of \eqref{change 21} is solved by
 \begin{align*}
  &2\varepsilon (2\mu+\lambda)\gamma^{-1} \int_0^t\int_{\mathbb{R}_+^2} p^{-1}\partial_t\mathcal{Z}^\alpha\partial_y p\cdot \mathcal{Z}^\alpha \partial_y p \ d{\bf{x}}ds\notag\\
  =&\varepsilon (2\mu+\lambda)\gamma^{-1}\left(\frac{d}{dt}\int_0^t\int_{\mathbb{R}_+^2}p^{-1}|\mathcal{Z}^\alpha \partial_y p |^2\ d{\bf{x}}ds- \int_0^t\int_{\mathbb{R}_+^2}\partial_t p^{-1}|\mathcal{Z}^\alpha \partial_y p |^2\ d{\bf{x}}ds\right)\notag\\
  \ge &\varepsilon (2\mu+\lambda)\gamma^{-1}\int_{\mathbb{R}_+^2}p^{-1}|\mathcal{Z}^\alpha \partial_y p |^2\;d{\bf{x}} -\varepsilon (2\mu+\lambda)\gamma^{-1}\int_{\mathbb{R}_+^2}p_0^{-1}|\mathcal{Z}^\alpha \partial_y p_0 |^2\;d{\bf{x}} \notag\\
  &-C\varepsilon  \| p^{-1}\|_{ 1,\infty}\int_0^t\|\partial_y p\|_{m-1}^2 \;ds.
 \end{align*}
 Then we estimate the second term on the right hand side of \eqref{change 20} in the following way.
 \begin{align*}
  &2\varepsilon (2\mu+\lambda) \int_0^t\int_{\mathbb{R}_+^2}\mathcal{Z}^\alpha [p^{-1} \partial_y p(\partial_x v_1+\partial_y v_2)]\cdot \mathcal{Z}^\alpha \partial_y p \ d{\bf{x}}ds\notag\\
  \lesssim&\varepsilon\Big[\sum_{\beta+\gamma+\iota=\alpha\atop |\gamma|, |\iota|\leq  |\beta|}\sup_{0\leq s\leq t}\|\mathcal{Z}^\gamma \partial_y p(s)\|_{L_{x}^\infty L_y^2}\|(\partial_x \mathcal{Z}^\iota v_1, \mathcal{Z}^\iota\partial_y v_2)\|_{L_{t, {\bf{x}}}^\infty}\left(\int_0^t\|\mathcal{Z}^\beta p^{-1}\|_{L_x^2L_y^\infty}^2\; ds\right)^\frac12\notag\\
  &+ \sum_{\beta+\gamma+\iota=\alpha\atop  |\beta|, |\iota|\leq |\gamma|}\|\mathcal{Z}^\beta p^{-1}\|_{L_{t, {\bf{x}}}^\infty  }\|(\partial_x \mathcal{Z}^\iota v_1, \mathcal{Z}^\iota\partial_y v_2)\|_{L_{t, {\bf{x}}}^\infty}\left(\int_0^t\|\mathcal{Z}^\gamma \partial_y p\|_{L_x^2 L_y^2}^2\; ds\right)^\frac12\notag\\
  &+ \sum_{\beta+\gamma+\iota=\alpha\atop  |\beta|, |\gamma|\leq |\iota|}\|\mathcal{Z}^\beta p^{-1}\|_{L_{t, {\bf{x}}}^\infty  }\sup_{0\leq s\leq t}\|\mathcal{Z}^\gamma \partial_y p(s)\|_{L_{x}^\infty L_y^2}\left(\int_0^t\|(\partial_x \mathcal{Z}^\iota v_1, \mathcal{Z}^\iota\partial_y v_2)\|_{L_{x}^2 L_y^\infty}^2\; ds\right)^\frac12\Big]\notag\\
  &\cdot\left(\int_0^t\|\mathcal{Z}^\alpha\partial_y  p\|_{L_x^2L_y^2}^2\; ds\right)^\frac12\notag\\
   \lesssim&\varepsilon \|(\partial_x v_1, \partial_y v_2, p^{-1})\|_{[(m-1)/2], \infty}\left[\| \partial_y p(0)\|_{[(m-1)/2]+2}+\left(\int_0^t\|\partial_y p\|_{[(m-1)/2]+2}^2 \;ds\right)^\frac12\right]\notag\\
  &\cdot\left(\int_0^t\|(p^{-1}-1, \partial_x v_1, \partial_y v_2)\|_{m-1}^2\; ds\right)^\frac14 \left(\int_0^t\|\partial_y(p^{-1}, \partial_x v_1, \partial_y v_2)\|_{m-1}^2\; ds\right)^\frac14 \notag\\
  &\cdot\left(\int_0^t\| \partial_y  p\|_{m-1}^2 ds\right)^\frac12 +  \varepsilon\|p^{-1}\|_{[(m-1)/2], \infty}\|(\partial_x v_1, \partial_y v_2)\|_{[(m-1)/2], \infty} \int_0^t\|\partial_y p\|_{m-1}^2 \;ds\\
    \lesssim&\varepsilon^2 \|(\partial_x v_1, \partial_y v_2, p^{-1})\|_{[(m-1)/2], \infty}^2\left(
    \| \partial_y p(0)\|_{[(m-1)/2]+2}^2+\int_0^t\|\partial_y p\|_{[(m-1)/2]+2}^2\; ds\right)\notag\\
    &\cdot \left( \int_0^t\| (p^{-1}-1, \partial_x v_1, \partial_y v_2)\|_{m-1}^2 ds\right)^\frac12  \left( \int_0^t\| \partial_y(p^{-1}-1, \partial_x v_1, \partial_y v_2)\|_{m-1}^2 ds\right)^\frac12\\
    &+  \varepsilon\|p^{-1}\|_{[(m-1)/2], \infty}\|(\partial_x v_1, \partial_y v_2)\|_{[(m-1)/2], \infty} \int_0^t\|\partial_y p\|_{m-1}^2 \;ds+\frac18\int_0^t \|\partial_yp\|_{m-1}^2\;ds.
 \end{align*}
 The third term on the right hand side of \eqref{change 20} can be bounded by
  \begin{align*}
  &2\varepsilon (2\mu+\lambda) \int_0^t\int_{\mathbb{R}_+^2}\partial_x\mathcal{Z}^\alpha\partial_y v_1\cdot \mathcal{Z}^\alpha \partial_y p \; d{\bf{x}}ds\notag\\
  \leq &2\varepsilon (2\mu+\lambda)\left(\int_0^t\|\partial_y v_1\|_{m}^2\; ds\right)^\frac12 \left(\int_0^t\|\partial_y p\|_{m-1}^2\; ds\right)^\frac12.
  \end{align*}
 By Lemma \ref{lem 2.2}, we handle the fourth term on the right hand side of \eqref{change 20}  by
 \begin{align*}
 &2\varepsilon (2\mu+\lambda)\gamma^{-1} \int_0^t\int_{\mathbb{R}_+^2}\mathcal{Z}^\alpha(p^{-1}\partial_y v_1\partial_x p)\cdot \mathcal{Z}^\alpha \partial_y p \ d{\bf{x}}ds\\
\le& C\varepsilon^2\|\partial_yv_1\|_{L_{t, {\bf{x}}}^\infty}^2\int_0^t \|p^{-1}\partial_xp\|_{m-1}^2\;ds+C\varepsilon^2\|p^{-1}\partial_xp\|_{L_{t, {\bf{x}}}^\infty}^2\int_0^t \|\partial_yv_1\|_{m-1}^2\;ds+\frac18\int_0^t\| \partial_y p\|_{m-1}^2 ds\\
\le& C\varepsilon^2\|\partial_yv_1\|_{L_{t, {\bf{x}}}^\infty}^2\|(p^{-1}, p)\|_{1, \infty}^2\int_0^t \|(p-1, p^{-1}-1)\|_{m}^2\;ds\\
&+C\varepsilon^2\|(p, p^{-1})\|_{1, \infty}^4\int_0^t \|\partial_yv_1\|_{m-1}^2\;ds+\frac18\int_0^t\| \partial_y p\|_{m-1}^2 \;ds.
 \end{align*}
  For the fifth term on the right hand side of \eqref{change 20}, we divide it into two parts.
  \begin{align}\label{bb}
    &2\varepsilon (2\mu+\lambda)\gamma^{-1} \int_0^t\int_{\mathbb{R}_+^2}\mathcal{Z}^\alpha(p^{-1} v_1\partial_x \partial_yp)\cdot \mathcal{Z}^\alpha \partial_y p \ d{\bf{x}}ds\notag\\
 =&   2\varepsilon (2\mu+\lambda)\gamma^{-1}\sum\limits_{|\beta|\ge 1} \int_0^t\int_{\mathbb{R}_+^2}\mathcal{Z}^\beta (\phi^{-1}v_1)\mathcal{Z}^{\gamma}(p^{-1}\phi\partial_x \partial_yp)\cdot \mathcal{Z}^\alpha \partial_y p \ d{\bf{x}}ds\notag\\
& +2\varepsilon (2\mu+\lambda)\gamma^{-1} \int_0^t\int_{\mathbb{R}_+^2}p^{-1} v_1\mathcal{Z}^\alpha\partial_x \partial_yp\cdot \mathcal{Z}^\alpha \partial_y p \ d{\bf{x}}ds.
    \end{align}
  The first part can be estimated as the fourth term above, and they have the similar bound.
     \begin{align*}
    & 2\varepsilon (2\mu+\lambda)\gamma^{-1}\sum\limits_{|\beta|\ge 1} \int_0^t\int_{\mathbb{R}_+^2}\mathcal{Z}^\beta(\phi^{-1} v_1)\mathcal{Z}^{\gamma}(p^{-1}\phi\partial_x \partial_yp)\cdot \mathcal{Z}^\alpha \partial_y p \ d{\bf{x}}ds\\
\le&C \varepsilon^2\|\partial_yv_1\|_{1, \infty}^2\| (p^{-1}, p)\|_{1, \infty}^2\int_0^t \|(p-1, p^{-1}-1)\|_{m-1}^2\;ds+C\varepsilon^2\|(p, p^{-1})\|_{1, \infty}^4\int_0^t \|\partial_yv_1\|_{m-1}^2\;ds\notag\\
&+\frac18\int_0^t\| \partial_y p\|_{m-1}^2 \;ds.
 \end{align*}
 As for the second part of \eqref{bb}, by integration by parts, we have
 \begin{align*}
 &2\varepsilon (2\mu+\lambda)\gamma^{-1} \int_0^t\int_{\mathbb{R}_+^2}p^{-1} v_1\mathcal{Z}^\alpha\partial_x \partial_yp\cdot \mathcal{Z}^\alpha \partial_y p \ d{\bf{x}}ds\\
 =&-2\varepsilon (2\mu+\lambda)\gamma^{-1} \int_0^t\int_{\mathbb{R}_+^2}\partial_x(p^{-1} v_1)\left(\mathcal{Z}^\alpha \partial_yp\right)^2 \;d{\bf{x}}ds\\
 \lesssim&\varepsilon\|(p^{-1}, v_1)\|_{1, \infty}^2\int_0^t \|\partial_yp\|_{m-1}^2\;ds.
 \end{align*}
 For the sixth term on the right hand side of \eqref{change 20}, we have
  \begin{align*}
  &2\varepsilon (2\mu+\lambda)\gamma^{-1} \int_0^t\int_{\mathbb{R}_+^2}\mathcal{Z}^\alpha(p^{-1}\partial_y v_2  \partial_yp)\cdot \mathcal{Z}^\alpha \partial_y p \ d{\bf{x}}ds\notag\\
  \lesssim&\varepsilon\Big[\sum_{\beta+\gamma+\iota=\alpha\atop   |\gamma|,|\iota|\leq |\beta|}\|\mathcal{Z}^\gamma\partial_y v_2\|_{L_{t, {\bf{x}}}^\infty}\sup_{0\leq s\leq t}\|\mathcal{Z}^\iota\partial_y p(s)\|_{L_{x}^\infty L_y^2}\left(\int_0^t \|\mathcal{Z}^\beta p^{-1}\|_{L_x^2 L_y^\infty}^2 ds\right)^\frac12 \notag\\
  &+ \sum_{\beta+\gamma+\iota=\alpha\atop |\beta|,|\iota| \leq  |\gamma|}\|\mathcal{Z}^\beta p^{-1}\|_{L_{t, {\bf{x}}}^\infty}\sup_{0\leq s\leq t}\|\mathcal{Z}^\iota\partial_y p(s)\|_{L_{x}^\infty L_y^2}\left(\int_0^t \|\mathcal{Z}^\gamma\partial_y v_2\|_{L_x^2 L_y^\infty}^2 ds\right)^\frac12 \notag\\
  &+ \sum_{\beta+\gamma+\iota=\alpha\atop |\beta|, |\gamma|\leq |\iota|<|\alpha|}\|\mathcal{Z}^\beta p^{-1}\|_{L_{t, {\bf{x}}}^\infty}\|\mathcal{Z}^\gamma \partial_y v_2\|_{L_{t, {\bf{x}}}^\infty  }\left(\int_0^t \|\mathcal{Z}^\iota\partial_y p\|_{L_x^2 L_y^2} ^2  ds\right)^\frac12\Big]\notag\cdot\left(\int_0^t\|\mathcal{Z}^\alpha \partial_y p\|_{L_x^2 L_y^2} ^2  ds\right)^\frac12\notag\\
  \le&\frac18\int_0^t\| \partial_y p\|_{m-1}^2\; ds+\frac{\varepsilon^2}{8}\int_0^t\|\partial_y^2 v_2\|_{m-1}^2 \;ds +C\varepsilon^2\|  \partial_y v_2\|_{[(m-1)/2], \infty}^2
  \left(\|\partial_y p(0)\|_{[(m-1)/2]+2}^2\right.\notag\\
  &\left.+\int_0^t \|\partial_y p\|_{[(m-1)/2]+2}^2\;ds \right) \cdot\sum_{j=0}^1\int_0^t\| \partial_y^j (p^{-1}-1)\|_{m-1}^2\; ds\notag\\
  &+\varepsilon^2\|p\|_{[(m-1)/2], \infty}^4\left[ \|\partial_y p(0)\|_{[(m-1)/2]+2}^4+\left(\int_0^t\|\partial_y p\|_{[(m-1)/2]+2}^2\ ds\right)^2\right]\int_0^t\|\partial_y v_2\|_{m-1}^2 \;ds.
  \end{align*}
 Next we consider the last term on the right hand side of \eqref{change 20} and separate it into three terms.
 \begin{align}\label{cc}
& 2\varepsilon (2\mu+\lambda)\gamma^{-1} \int_0^t\int_{\mathbb{R}_+^2}\mathcal{Z}^\alpha(p^{-1}v_2  \partial_y^2p)\cdot \mathcal{Z}^\alpha \partial_y p \ d{\bf{x}}ds\notag\\
=& 2\varepsilon (2\mu+\lambda)\gamma^{-1} \int_0^t\int_{\mathbb{R}_+^2}p^{-1}v_2\partial_y\mathcal{Z}^\alpha\partial_yp\cdot \mathcal{Z}^\alpha \partial_y p \ d{\bf{x}}ds\notag\\
&+ 2\varepsilon (2\mu+\lambda)\gamma^{-1} \int_0^t\int_{\mathbb{R}_+^2}p^{-1}v_2[\mathcal{Z}^\alpha, \partial_y]\partial_yp\cdot \mathcal{Z}^\alpha \partial_y p \ d{\bf{x}}ds\notag\\
&+2\varepsilon (2\mu+\lambda)\gamma^{-1}\sum\limits_{|\beta|\ge 1} \int_0^t\int_{\mathbb{R}_+^2}\mathcal{Z}^\beta(p^{-1}v_2 )\mathcal{Z}^\gamma \partial_y^2p\cdot \mathcal{Z}^\alpha \partial_y p \;d{\bf{x}}ds.
 \end{align}
 The first one can be estimated by integration by parts.
 \begin{align*}
 & 2\varepsilon (2\mu+\lambda)\gamma^{-1} \int_0^t\int_{\mathbb{R}_+^2}p^{-1}v_2\partial_y\mathcal{Z}^\alpha\partial_yp\cdot \mathcal{Z}^\alpha \partial_y p \ d{\bf{x}}ds\\
 \lesssim& \varepsilon\left(\|\phi\partial_yp^{-1}\phi^{-1}v_2\|_{L_{t, {\bf{x}}}^\infty}+\| p^{-1}\partial_yv_2\|_{L_{t, {\bf{x}}}^\infty}\right)\int_0^t \|\partial_y p\|_{m-1}^2\;ds\\
 \lesssim& \varepsilon\|p^{-1}\|_{1, \infty}\|\partial_yv_2\|_{L_{t, {\bf{x}}}^\infty}\int_0^t \|\partial_y p\|_{m-1}^2\;ds.
 \end{align*}
 Then we estimate the second part of \eqref{cc}. By Lemma \ref{L2}, one has
 \begin{align*}
 &2\varepsilon (2\mu+\lambda)\gamma^{-1} \int_0^t\int_{\mathbb{R}_+^2}p^{-1}v_2[\mathcal{Z}^\alpha, \partial_y]\partial_yp\cdot \mathcal{Z}^\alpha \partial_y p \ d{\bf{x}}ds\\
 \lesssim&\varepsilon  \sum_{k=0}^{m-2}\int_0^t\int_{\mathbb{R}_+^2}p^{-1}v_2\phi^{k,m}(y)\partial_y\mathcal{Z}_y^k\partial_yp\cdot \mathcal{Z}^\alpha \partial_y p \ d{\bf{x}}ds\\
 \lesssim&\varepsilon \|p^{-1}\|_{L_{t, {\bf{x}}}^\infty}\|\partial_yv_2\|_{L_{t, {\bf{x}}}^\infty}\int_0^t \|\partial_yp\|_{m-1}^2\;ds.
 \end{align*}
 By using the Hardy trick, we find the last part of \eqref{cc} can be estimated as the sixth term of \eqref{change 20}.
 \begin{align*}
 &2\varepsilon (2\mu+\lambda)\gamma^{-1}\sum\limits_{|\beta|\ge 1} \int_0^t\int_{\mathbb{R}_+^2}\mathcal{Z}^\beta(p^{-1}v_2 )\mathcal{Z}^\gamma \partial_y^2p\cdot \mathcal{Z}^\alpha \partial_y p \;d{\bf{x}}ds\\
 \lesssim &\varepsilon \sum\limits_{|\beta|\ge 1} \int_0^t\int_{\mathbb{R}_+^2}\phi^{-1}\mathcal{Z}^\beta(p^{-1}v_2 )\phi\mathcal{Z}^\gamma \partial_y^2p\cdot \mathcal{Z}^\alpha \partial_y p \;d{\bf{x}}ds\\
 \le& \frac18\int_0^t\| \partial_y p\|_{m-1}^2 ds+ \frac{\varepsilon^2}{8}\int_0^t\| \partial_y^2 v_2\|_{m-1}^2 ds +C\varepsilon^2\|p^{-1}\|_{L_{t, {\bf{x}}}^\infty}^2\|\partial_y v_2\|_{L_{t, {\bf{x}}}^\infty}^2\int_0^t\|\partial_y p\|_{m-1}^2\; ds\notag\\
 &+C\varepsilon^2\|\partial_y v_2\|_{L_{t, {\bf{x}} }^\infty}^2\left(\|\partial_y p(0)\|_{2}^2+\int_0^t\|\partial_y p\|_{2}^2 \;ds\right)\left( \int_0^t\| p^{-1}-1\|_{m-1}^2\; ds\right)^\frac12\left( \int_0^t\|\partial_y p\|_{m-1}^2 \;ds\right)^\frac12\notag\\
&+C\varepsilon^4\| p^{-1}\|_{L_{t, {\bf{x}}}^\infty}^4\left(\|\partial_y p(0)\|_{2}^4+\int_0^t\|\partial_y p\|_{2}^4 \;ds\right)  \int_0^t\|  \partial_y v_2\|_{m-1}^2 ds.
 \end{align*}
  Finally, for the terms on the right hand side of \eqref{change 12}, we can get the following bound directly.
  \begin{align*}
  &\Big(1+\|(\rho, {\bf{v}}, b_1, \tilde{b}_2, \partial_y v_2)\|_{1, \infty}^2\Big)^2\int_0^t\|(\rho,  {\bf{v}},  b_1, \tilde{b}_2)\|_{m}^2 \;ds+\varepsilon^2\mu^2\int_0^t\|\partial_x v_2\|_{m}^2\; ds\notag\\
  &+\varepsilon^2(\mu+\lambda)^2\int_0^t\|\partial_y v_1\|_{m}^2 \;ds +\|b_1\|_{L_{t, {\bf{x}}}^\infty}^2\int_0^t\|\partial_y b_1\|_{m-1}^2 \;ds\notag\\
  &+\left( \|\partial_y b_1(0)\|_{2}^2+\int_0^t\|\partial_y b_1\|_{2}^2\; ds\right)\left(\int_0^t\| b_1\|_{m-1}^2 ds+\int_0^t\|    \partial_y b_1\|_{m-1}^2 \;ds\right).
  \end{align*}
Summing up all the estimates together, we conclude by induction that
\begin{align*}
&\int_0^t\int_{\mathbb{R}_+^2}\left( |\mathcal{Z}^\alpha \partial_y p |^2+ \varepsilon ^2(2\mu+\lambda)^2 |\mathcal{Z}^\alpha \partial_y^2 v_2|^2\right)  \ d{\bf{x}}ds+\varepsilon (2\mu+\lambda)\gamma^{-1}\int_{\mathbb{R}_+^2}p^{-1}|\mathcal{Z}^\alpha \partial_y p |^2\;ds \\
\lesssim &\varepsilon (2\mu+\lambda)\gamma^{-1}\int_{\mathbb{R}_+^2}p_0^{-1}|\mathcal{Z}^\alpha \partial_y p_0 |^2\;ds+\varepsilon^2\int_0^t\|\partial_x v_2\|_{m}^2\; ds+\varepsilon^2\int_0^t\|\partial_y v_1\|_{m}^2 \;ds\notag\\
&+\left(1+\|(p, p^{-1},  {\bf{v}}, {\bf{B}}-\overset{\rightarrow}{e_y}, \partial_y v_2)\|_{[(m-1)/2], \infty}^2+\varepsilon^2\|\partial_y v_1\|_{L_{t, {\bf{x}}}^\infty} ^2\right)^2\int_0^t\|(p-1, \rho-1, {\bf{v}},  {\bf{B}}-\overset{\rightarrow}{e_y})\|_{m}^2 \;ds \notag\\
&+\|b_1\|_{L_{t, {\bf{x}}}^\infty}^2\int_0^t\|\partial_y b_1\|_{m-1}^2 \;ds +\varepsilon^2\left(1+\|(p^{-1}, v_1, \partial_y v_2, p)\|_{[(m-1)/2]+1, \infty}^2\right)^2\int_0^t\|(\partial_y p, \partial_y v_1)\|_{m-1}^2 \;ds\notag\\
&+\varepsilon^2\left(1+\|(p, p^{-1}, \partial_x v_1, \partial_y v_2)\|_{[(m-1)/2], \infty}\right)^2\left(1+\|(\partial_y p, \partial_y b_1)(0)\|_{[(m-1)/2]+3}^2\right.\notag\\
&\left.+\int_0^t\|(\partial_y p, \partial_y b_1)\|_{[(m-1)/2]+3}^2\; ds\right)^2 \cdot\sum_{j=0}^1\int_0^t\|\partial_y^j( \partial_x  v_1,  v_2, b_1,   p^{-1}-1,  p-1 )\|_{m-1}^2 \; ds.
\end{align*}

\subsection{Conormal Estimate of $  \partial_y^2 v_2$}   The conormal estimates of $\partial_y^2 v_2$ are derived in this part to control $\|\partial_yv_2\|_{L_{t, {\bf{x}}}^\infty}$.  Similar  to  \eqref{change 8}, we have
\begin{align}\label{4.25}
\int_0^t\int_{\mathbb{R}_+^2}|\mathcal{Z}^\alpha\partial_y^{2}  v_2|^2 \;d{\bf{x}}ds
\lesssim& \int_0^t\int_{\mathbb{R}_+^2}|\mathcal{Z}^\alpha\partial_y \partial_x v_1|^2\;d{\bf{x}}ds+ \int_0^t\int_{\mathbb{R}_+^2}\left|\mathcal{Z}^\alpha\partial_y ( p^{-1}\partial_t p)\right|^2\;d{\bf{x}}ds\notag\\
&+ \int_0^t\int_{\mathbb{R}_+^2}\left| \mathcal{Z}^\alpha\partial_y (p^{-1}{\bf{v}}\cdot\nabla p)\right|^2 d{\bf{x}}ds.
\end{align}
For the first and the second terms on the right hand side of \eqref{4.25}, one has
\begin{align*}
\int_0^t\int_{\mathbb{R}_+^2}|\mathcal{Z}^\alpha\partial_y \partial_x v_1|^2\; d{\bf{x}}ds\lesssim\int_0^t \|\partial_y  v_1\|_{m-1}^2\; ds,
\end{align*}
and
\begin{align*}
&\int_0^t\int_{\mathbb{R}_+^2}|\mathcal{Z}^\alpha\partial_y (  p^{-1}\partial_t p)|^2\; d{\bf{x}}ds\notag\\
\lesssim  &\sup_{0\leq s\leq t}\|(\partial_y p^{-1}, \partial_t\partial_y p)(s)\|_{L_{  x}^\infty L_y^2}^2\int_0^t \|(\mathcal{Z}^\alpha\partial_t p,  \mathcal{Z}^\alpha p^{-1})\|_{L_{x}^2L_y^\infty}^2  \;ds\notag\\
&+\|(\partial_t p, p^{-1})\|_{L_{t, {\bf{x}}}^\infty}^2\int_0^t \|(\mathcal{Z}^\alpha \partial_y p^{-1}, \mathcal{Z}^\alpha \partial_t\partial_y p)\|_{L_x^2 L_y^2}^2\;ds\notag\\
\lesssim&\left(\|(\partial_y p^{-1}, \partial_y p)(0)\|_{[m/2]+2}^2+\int_0^t\|(\partial_y p^{-1}, \partial_y p)\|_{[m/2]+2}^2\; ds\right)\left(\int_0^t\|(p-1, p^{-1}-1)\|_{m-1}^2\; ds\right)^\frac12\notag\\
&\cdot\left(\int_0^t\|\partial_y(p, p^{-1})\|_{m-1}^2\; ds\right)^\frac12+\|(p, p^{-1})\|_{1, \infty}^2\int_0^t\|(\partial_y p, \partial_y p^{-1})\|_{m-1}^2 \;ds.
\end{align*}
 Then we write  the third term on the right hand side of \eqref{4.25} as
\begin{equation}\label{4.26}
\begin{split}
&\int_0^t\int_{\mathbb{R}_+^2}| \mathcal{Z}^\alpha\partial_y (p^{-1}{\bf{v}}\cdot \nabla p)|^2\;d{\bf{x}}ds\\
\lesssim&\int_0^t\int_{\mathbb{R}_+^2}| \mathcal{Z}^\alpha(\partial_y p^{-1} v_1 \partial_x p)|^2\;d{\bf{x}}ds+\int_0^t\int_{\mathbb{R}_+^2}| \mathcal{Z}^\alpha(\partial_y p^{-1} v_2 \partial_y p)|^2\;d{\bf{x}}ds\\
&+\int_0^t\int_{\mathbb{R}_+^2}| \mathcal{Z}^\alpha  (p^{-1}\partial_y{\bf{v}}\cdot\nabla p)|^2\;d{\bf{x}}ds+\int_0^t\int_{\mathbb{R}_+^2}| \mathcal{Z}^\alpha( p^{-1}v_1\partial_x\partial_y p)|^2\;d{\bf{x}}ds\\
&+\int_0^t\int_{\mathbb{R}_+^2}| \mathcal{Z}^\alpha  (p^{-1}v_2\partial_y^2 p)|^2\;d{\bf{x}}ds.
\end{split}
\end{equation}
The first term on the right hand side of \eqref{4.26} is dealt by
\begin{align*}
&\int_0^t\int_{\mathbb{R}_+^2}| \mathcal{Z}^\alpha(\partial_y p^{-1} v_1 \partial_x p)|^2\;d{\bf{x}}ds\notag\\
\lesssim&\sum_{\beta+\gamma+\iota=\alpha\atop   |\gamma|,|\iota|\leq |\beta|}\|\mathcal{Z}^\gamma v_1\|_{L_{t, {\bf{x}}}^\infty}^2\|\mathcal{Z}^\iota \partial_x p\|_{L_{t, {\bf{x}}}^\infty}^2\int_0^t\|\mathcal{Z}^\beta \partial_y p^{-1}\|_{L_x^2L_y^2}^2 \;ds\notag\\
&+\sum_{\beta+\gamma+\iota=\alpha\atop  |\beta|,|\iota|\leq  |\gamma|}\|\phi\mathcal{Z}^\beta \partial_y p^{-1}\|_{L_{t,{\bf{x}}} ^\infty  }^2\|\mathcal{Z}^\iota \partial_x p\|_{L_{t, {\bf{x}}}^\infty}^2\int_0^t\|\phi^{-1}\mathcal{Z}^\gamma v_1\|_{L_x^2L_y^2}^2 \;ds\notag\\
&+\sum_{\beta+\gamma+\iota=\alpha\atop   |\beta|,|\gamma|\leq |\iota|}\sup_{0\leq s\leq t}\| \mathcal{Z}^\beta \partial_y p^{-1}(s)\|_{L_{x} ^\infty L_y^2 }^2\| \mathcal{Z}^\gamma v_1\|_{L_{t, {\bf{x}}}^\infty}^2\int_0^t\|\mathcal{Z}^\iota \partial_x p\|_{L_{x}^2L_y^\infty}^2\; ds\notag\\
\lesssim&\|(v_1, p^{-1})\|_{[m/2], \infty}^2\|p\|_{[m/2], \infty}^2\int_0^t\|(\partial_y v_1, \partial_y p^{-1})\|_{m-2}^2 \;ds\notag\\
&+\|v_1\|_{[m/2], \infty}^2\left( \|\partial_y p^{-1}(0)\|_{[m/2]+ 2}^2+\int_0^t\|\partial_y p^{-1}\|_{[m/2]+2}^2\; ds\right)\notag\\
&\cdot\left(\int_0^t\| p-1\|_{m-1}^2\; ds\right)^\frac12\left(\int_0^t\| \partial_y p\|_{m-1}^2\; ds\right)^\frac12.
\end{align*}

 By the Sobolev embedding inequality, we estimate the second term on the right hand side of \eqref{4.26}   by
 \begin{align*}
 &\int_0^t\int_{\mathbb{R}_+^2}| \mathcal{Z}^\alpha(\partial_y p^{-1} v_2 \partial_y p)|^2\;d{\bf{x}}ds\notag\\
 \lesssim&\sum_{\beta+\gamma+\iota=\alpha\atop  |\gamma|,|\iota|\leq  |\beta|}\|\phi^{-1}\mathcal{Z}^\gamma v_2\|_{L_{t, {\bf{x}}}^\infty}^2\|\phi\mathcal{Z}^\iota \partial_y p\|_{L_{t, {\bf{x}}}^\infty}^2\int_0^t\|\mathcal{Z}^\beta \partial_y p^{-1}\|_{L_x^2 L_y^2}^2 \;ds\notag\\
 &+\sum_{\beta+\gamma+\iota=\alpha\atop   |\beta|,|\iota|\leq |\gamma|}\|\phi\mathcal{Z}^\beta \partial_y p^{-1}\|_{L_{t,{\bf{x}}} ^\infty  }^2\sup_{0\leq s\leq t}\|\mathcal{Z}^\iota \partial_y p(s)\|_{L_{x}^\infty L_y^2}^2\int_0^t\|\phi^{-1}\mathcal{Z}^\gamma v_2\|_{L_{x}^2L_y^\infty}^2\; ds\notag\\
 &+\sum_{\beta+\gamma+\iota=\alpha\atop  |\beta|,|\gamma|\leq  |\iota|}\|\phi \mathcal{Z}^\beta \partial_y p^{-1}\|_{L_{t, {\bf{x}}} ^\infty   }^2\|\phi^{-1} \mathcal{Z}^\gamma v_2\|_{L_{t, {\bf{x}}}^\infty}^2\int_0^t\|\mathcal{Z}^\iota \partial_y p\|_{L_x^2 L_y^2}^2 \;ds\notag\\
 \le&C\|(p^{-1}, \partial_y v_2)\|_{[m/2], \infty}^2\|(p, \partial_y v_2)\|_{[m/2], \infty}^2\int_0^t\|(\partial_y  p^{-1}, \partial_y p)\|_{m-2}^2 \;ds\notag\\
 &+ \|p^{-1}\|_{[m/2], \infty}^4\left(  \|\partial_y  p(0)\|_{[m/2]+1}^2+\int_0^t\|\partial_y  p\|_{[m/2]+1}^2\; ds\right)^2\int_0^t\|\partial_y v_2\|_{m-2}^2 \;ds\notag\\
 &+\frac16\int_0^t\|\partial_y^2 v_2\|_{m-2}^2 \;ds.
 \end{align*}
 The third term on the right hand side of \eqref{4.26} is handled by
 \begin{align*}
 &\int_0^t\int_{\mathbb{R}_+^2}| \mathcal{Z}^\alpha  (p^{-1}\partial_y{\bf{v}}\cdot\nabla p)|^2\; d{\bf{x}}ds\notag\\
 \lesssim&\sum_{\beta+\gamma+\iota=\alpha\atop
  |\gamma|, |\iota|\leq |\beta|}\sup_{0\leq s\leq t}\|\mathcal{Z}^\gamma \partial_y v_1(s)\|_{L_{x}^\infty L_y^2}^2\|\mathcal{Z}^\iota \partial_x p\|_{L_{t, {\bf{x}}}^\infty}^2\int_0^t\|\mathcal{Z}^\beta p^{-1}\|_{L_{x}^2L_y^\infty}^2 \; ds\notag\\
&+\sum_{\beta+\gamma+\iota=\alpha\atop
  |\gamma|, |\iota|\leq |\beta|}\|\mathcal{Z}^\gamma \partial_y v_2\|_{L_{t,{\bf{x}} }^\infty  }^2\sup_{0\leq s\leq t}\|\mathcal{Z}^\iota \partial_y p(s)\|_{L_{ x}^\infty L_y^2}^2\int_0^t\|\mathcal{Z}^\beta p^{-1}\|_{L_{x}^2L_y^\infty}^2 \;ds\notag\\
&+\sum_{\beta+\gamma+\iota=\alpha\atop
  |\beta|, |\iota|\leq |\gamma|}\|\mathcal{Z}^\beta p^{-1}\|_{L_{t,{\bf{x}} }^\infty }^2\|\mathcal{Z}^\iota \partial_x p\|_{L_{t, {\bf{x}}}^\infty}^2\int_0^t\|\mathcal{Z}^\gamma \partial_y v_1\|_{L_x^2L_y^2}^2 \;ds\notag\\
&+\sum_{\beta+\gamma+\iota=\alpha\atop
  |\beta|, |\iota|\leq 	|\gamma|}\|\mathcal{Z}^\beta p^{-1}\|_{L_{t,{\bf{x}} }^\infty }^2\sup_{0\leq s\leq t}\|\mathcal{Z}^\iota \partial_y p(s)\|_{L_{x}^\infty L_y^2}^2\int_0^t\|\mathcal{Z}^\gamma \partial_y v_2\|_{L_{x}^2L_y^\infty}^2 \;ds\notag\\
&+\sum_{\beta+\gamma+\iota=\alpha\atop
  |\beta|, |\gamma|\leq 	|\iota|}\|\mathcal{Z}^\beta p^{-1}\|_{L_{t,{\bf{x}} }^\infty }^2\sup_{0\leq s\leq t}\|\mathcal{Z}^\gamma \partial_y v_1(s)\|_{L_{x}^\infty L_y^2}^2\int_0^t\|\mathcal{Z}^\iota \partial_x p\|_{L_{x}^2L_y^\infty}^2 \;ds\notag\\
&+\sum_{\beta+\gamma+\iota=\alpha\atop
  |\beta|, |\gamma|\leq 	|\iota|}\|\mathcal{Z}^\beta p^{-1}\|_{L_{t,{\bf{x}} }^\infty }^2\|\mathcal{Z}^\gamma \partial_y v_2\|_{L_{t,  {\bf{x}}}^\infty  }^2\int_0^t\|\mathcal{Z}^\iota \partial_y p\|_{L_{x}^2L_y^2}^2 \;ds\notag\\
\lesssim&\|(p^{-1}, p, \partial_yv_2)\|_{[m/2], \infty}^2\left( \| (\partial_y v_1, \partial_y p)(0)\|_{[m/2]+1}^2+\int_0^t\|(\partial_y v_1, \partial_y p)\|_{[m/2]+1}^2 \;ds\right)\notag\\
&\cdot\left(\int_0^t\|(p^{-1}-1, \partial_y v_2, \partial_x p)\|_{m-2}^2\; ds\right)^\frac12\left(\int_0^t\|(\partial_y p^{-1}, \partial_y^2 v_2, \partial_x \partial_y p)\|_{m-2}^2\; ds\right)^\frac12\notag\\
&+\|p^{-1}\|_{[m/2], \infty}^2\|(p, \partial_y v_2)\|_{[m/2], \infty}^2\int_0^t\|(\partial_y v_1, \partial_y p)\|_{m-2}^2 \;ds.
 \end{align*}

For the fourth   term on the right hand side of \eqref{4.26}, we have
\begin{align*}
&\int_0^t\int_{\mathbb{R}_+^2}| \mathcal{Z}^\alpha( p^{-1}v_1\partial_x\partial_y p)|^2\;d{\bf{x}}ds\\
\lesssim&\sum_{\beta+\gamma+\iota=\alpha\atop  |\gamma|, |\iota|\leq |\beta|}\|\mathcal{Z}^\gamma v_1\|_{L_{t, {\bf{x}}}^\infty}^2\sup_{0\leq s\leq t}\|\mathcal{Z}^\iota \partial_x\partial_y p(s)\|_{L_{  x}^\infty L_y^2}^2\int_0^t\|\mathcal{Z}^\beta p^{-1}\|_{L_x^2L_y^\infty}^2\; ds\notag\\
&+\sum_{\beta+\gamma+\iota=\alpha\atop  |\beta|, |\iota|\leq |\gamma|}\|\mathcal{Z}^\beta p^{-1}\|_{L_{t, {\bf{x}}}^\infty}^2\sup_{0\leq s\leq t}\|\mathcal{Z}^\iota \partial_x\partial_y p(s)\|_{L_{ x}^\infty L_y^2}^2\int_0^t\|\mathcal{Z}^\gamma v_1\|_{L_x^2L_y^\infty}^2\; ds\notag\\
&+\sum_{\beta+\gamma+\iota=\alpha\atop |\beta|, |\gamma|\leq  |\iota|}\|\mathcal{Z}^\beta p^{-1}\|_{L_{t, {\bf{x}}}^\infty}^2\|\mathcal{Z}^\gamma v_1\|_{L_{t, {\bf{x}}}^\infty  }^2\int_0^t\|\mathcal{Z}^\iota \partial_x\partial_y p\|_{L_x^2L_y^2}^2 \;ds\notag\\
\lesssim& \|(v_1, p^{-1})\|_{[m/2], \infty}^2\left( \|\partial_y p(0)\|_{[m/2]+2}^2+\int_0^t\|\partial_y p\|_{[m/2]+2}^2 ds\right) \left(\int_0^t\|(p^{-1}-1, v_1)\|_{m-2}^2 ds\right)^\frac12\notag\\
&\cdot\left(\int_0^t\|(\partial_y p^{-1}, \partial_y v_1)\|_{m-2}^2 ds\right)^\frac12 +\| p^{-1}\|_{[m/2], \infty}^2\|  v_1\|_{[m/2], \infty}^2\int_0^t\|\partial_y p\|_{m-1}^2 \;ds.
\end{align*}

 The last term on the right hand side of \eqref{4.26} is dealt by
 \begin{align*}
 &\int_0^t\int_{\mathbb{R}_+^2}| \mathcal{Z}^\alpha  (p^{-1}v_2\partial_y^2 p)|^2\;d{\bf{x}}ds\notag\\
 \lesssim&\sum_{\beta+\gamma+\iota=\alpha \atop  |\gamma|, |\iota|\leq |\beta|}\|\phi^{-1}\mathcal{Z}^\gamma v_2\|_{L_{t, {\bf{x}}}^\infty}^2\sup_{0\leq s\leq t}\|\phi \mathcal{Z}^\iota \partial_y^2 p(s)\|_{L_{x}^\infty L_y^2}^2\int_0^t\|\mathcal{Z}^\beta p\|_{L_x^2L_y^\infty}^2\; ds\notag\\
 &+\sum_{\beta+\gamma+\iota=\alpha \atop |\beta|, |\iota|\leq  |\gamma|}\|\mathcal{Z}^\beta p\|_{L_{t, {\bf{x}}}^\infty}^2\sup_{0\leq s\leq t}\|\phi \mathcal{Z}^\iota \partial_y^2 p(s)\|_{L_{ x}^\infty L_y^2}^2\int_0^t\|\phi^{-1}\mathcal{Z}^\gamma v_2\|_{L_x^2L_y^\infty}^2\; ds\notag\\
 &+\sum_{\beta+\gamma+\iota=\alpha \atop |\beta|, |\gamma|\leq  |\iota|}\|\mathcal{Z}^\beta p\|_{L_{t, {\bf{x}}}^\infty}^2\|\phi^{-1}\mathcal{Z}^\gamma v_2\|_{L_{t, {\bf{x}}}^\infty }^2\int_0^t\|\phi \mathcal{Z}^\iota \partial_y^2 p\|_{L_x^2L_y^2 }^2 \;ds\notag\\
 \lesssim& \|(p, \partial_y v_2)\|_{[m/2], \infty}^2\left( \|\partial_y p(0)\|_{[m/2]+2}^2+\int_0^t\|\partial_y p\|_{[m/2]+2}^2 ds\right) \left(\int_0^t\|(p-1, \partial_y v_2)\|_{m-2}^2\; ds\right)^\frac12\notag\\
 &\cdot\left(\int_0^t\|(\partial_y p, \partial_y^2 v_2)\|_{m-2}^2\; ds\right)^\frac12 +\|p\|_{[m/2], \infty}^2\|\partial_y v_2\|_{[m/2], \infty}^2\int_0^t\|\partial_y p\|_{m-1}^2 \;ds.
 \end{align*}
 Thus, we conclude by H\"older's inequality that
 \begin{align*}
 \int_0^t \|\partial_y^2v_2\|_{m-2}^2\;ds
 \lesssim&\left(1+ \|({\bf{v}}, p, p^{-1}, \partial_y v_2)\|_{[m/2], \infty}^2\right)^2\int_0^t\|(\partial_y v_1, \partial_y p, \partial_y p^{-1})\|_{m-1}^2 \;ds\notag\\
 &+\left(1+\|({\bf{v}}, p, p^{-1}, \partial_y v_2)\|_{[m/2], \infty}^2\right)^2 \left(1+\| (\partial_y v_1, \partial_y p, \partial_y p^{-1})(0)\|_{ [m/2]+2}^2\right.\notag\\
 &\left.+\int_0^t\| (\partial_y v_1, \partial_y p, \partial_y p^{-1})\|_{[m/2]+2}^2\ ds\right)^2\cdot\sum_{j=0}^1\int_0^t\|\partial_y^j( {\bf{v}}, p-1, p^{-1}-1)\|_{m-1}^2 \;ds.
 \end{align*}

\section{Proof of Theorem $\ref{Th1}$}
 Now we prove Theorem \ref{Th1}. According to the  estimate in Section 3 and Section  4, one has
\begin{align}\label{5.1}
& N_m(t)+  \varepsilon(2\mu+\lambda)\gamma^{-1} \sum_{|\alpha|+i\leq  m\atop i=1,2}\int_{\mathbb{R}_+^2}
p^{-1}(t)|\mathcal{Z}^\alpha \partial_y^i p(t)|^2 \;d{\bf{x}} \notag\\
\lesssim&\;N_m(0)+ \varepsilon(2\mu+\lambda) \gamma^{-1} \sum_{|\alpha|+i\leq m\atop i=1,2}\int_{\mathbb{R}_+^2}
p_0^{-1}|\mathcal{Z}^\alpha \partial_y^i p_0|^2\; d{\bf{x}}\notag\\
&+\left\{ \left(1+\|(p, {\bf{v}}, {\bf{B}}-\overset{\rightarrow}{e_y}, \partial_yv_2)\|_{[m/2]+1, \infty}^2\right)^3+\left(1+\|(p, {\bf{v}}, \partial_y v_2)\|_{[m/2], \infty}^2\right)^2\right.\notag\\
&\left.\cdot\left(1+\|(\partial_y {\bf{v}}, \partial_y p, \partial_y {\bf{B}})(0)\|_{[(m-1)/2]+3}^2+\int_0^t\|(\partial_y {\bf{v}}, \partial_y p, \partial_y {\bf{B}})\|_{[(m-1)/2]+3}^2\;ds\right)^2\right.\notag\\
&\left.+\varepsilon\|(v_1, \partial_yv_1, \partial_yb_1)\|_{2, \infty}^2  \left(1+\|(\rho, p^{-1}, b_1, \partial_y v_2)\|_{L_{t, {\bf{x}}}^\infty}^2\right)\right\} \cdot\sum_{j=0}^1 \int_0^t \|\partial_y^j({\bf{v}}, {\bf{B}}-\overset{\rightarrow}{e_y}, p-1)\|_{m-j}^2  \; ds.
\end{align}
By Lemma \ref{lem 2.2}, we have
\begin{align*}
&\|(p, {\bf{v}}, {\bf{B}}-\overset{\rightarrow}{e_y}, \partial_yv_2)\|_{[m/2]+1, \infty}^2\\
\lesssim&\|(p-1, {\bf{v}}, {\bf{B}}-\overset{\rightarrow}{e_y}, \partial_yv_2)(0)\|_{[m/2]+3}^2
+\|\partial_y(p, {\bf{v}}, {\bf{B}}-\overset{\rightarrow}{e_y}, \partial_yv_2)\|_{[m/2]+3}^2\\
&+\int_0^t\left(	\|(p-1, {\bf{v}}, {\bf{B}}-\overset{\rightarrow}{e_y}, \partial_yv_2)\|_{[m/2]+4}^2+	\|\partial_y(p, {\bf{v}}, {\bf{B}}-\overset{\rightarrow}{e_y}, \partial_yv_2)\|_{[m/2]+3}^2\right) \;ds\\
\lesssim&\;\mathcal{P}(N_m(0))+t\mathcal{P}(N_m(t)).
\end{align*}
On the other hand, by Sobolev embedding, we also have
\begin{align*}
\varepsilon\|(\partial_yv_1, \partial_yb_1)\|_{2, \infty}^2
&\lesssim\varepsilon\sum_{|\alpha|\leq 2}\sup_{0\leq s\leq t}\|\mathcal{Z}^\alpha(\partial_yv_1, \partial_yb_1)(s)\|_{L_{x}^\infty L_y^2}\|\mathcal{Z}^\alpha(\partial_y^2v_1, \partial_y^2b_1)(s)\|_{L_{x}^\infty L_y^2}\\
&\lesssim  \varepsilon \left[\|(\partial_yv_1, \partial_yb_1)(0)\|_{4}+\left(\int_0^t\|(\partial_yv_1, \partial_yb_1)\|_{4}^2\; ds\right)^\frac12\right]\notag\\
&\cdot\left[\|(\partial_y^2v_1, \partial_y^2b_1)(0)\|_{4}+\left(\int_0^t\|(\partial_y^2v_1, \partial_y^2b_1)\|_{4}^2\; ds\right)^\frac12\right]\notag\\
\lesssim&\|(\partial_yv_1, \partial_yb_1)(0)\|_{4}^2+  \varepsilon^2\|(\partial_y^2v_1, \partial_y^2b_1)(0)\|_{4}^2+ \int_0^t\|(\partial_yv_1, \partial_yb_1)\|_{4}^2 \;ds \notag\\
&+\varepsilon^2\int_0^t\|(\partial_y^2v_1, \partial_y^2b_1)\|_{4}^2 \;ds
\end{align*}
Thus, for any $m\geq 9$,  by inserting the above inequalities into \eqref{5.1},   we obtain
\begin{equation*}
\begin{split}
& N_m(t)+ \varepsilon(2\mu+\lambda)\gamma^{-1} \sum_{|\alpha|+i\leq m\atop i=1,2}\int_{\mathbb{R}_+^2}
 p^{-1}(t)|\mathcal{Z}^\alpha \partial_y^i p(t)|^2 \; d{\bf{x}}\\
\lesssim&\;\;\mathcal{P}(N_m(0))+[t+\varepsilon(2\mu+\lambda)]\mathcal{P}(N_m(t)).
\end{split}
\end{equation*}
Let  the time $t$ and $\varepsilon$ be suitablely small,  then we achieve that
\begin{align*}
N_m(t)+ \varepsilon(2\mu+\lambda)\gamma^{-1} \sum_{|\alpha|+i\leq m\atop i=1,2}\int_{\mathbb{R}_+^2}
 p^{-1}(t)|\mathcal{Z}^\alpha \partial_y^i p(t)|^2 \;d{\bf{x}}
\lesssim\mathcal{P}(N_m(0)).
\end{align*}	
Based on the above uniform conormal estimates achieved, the inviscid limit in Theorem \ref{Th1} can be verified as in \cite{Mas-Rou1}.

	\end{document}